\documentclass{article}

\usepackage[a4paper,top=2cm,bottom=3cm,left=2.5cm,right=2.5cm,marginparwidth=1.75cm]{geometry}
\usepackage{graphicx}
\usepackage{multicol,multirow}
\usepackage{amsmath,amssymb,amsfonts}
\usepackage{mathrsfs}
\usepackage{rotating}
\usepackage{appendix}
\usepackage[numbers]{natbib}
\usepackage{amsthm}
\usepackage{lipsum}
\usepackage{todonotes}
\newtheorem{theorem}{Theorem}[section]
\newtheorem{proposition}[theorem]{Proposition}
\newtheorem{lemma}[theorem]{Lemma}
\newtheorem{corollary}[theorem]{Corollary}
\newtheorem{definition}{Definition}[section]
\newtheorem{remark}[theorem]{Remark}
\newtheorem{example}[theorem]{Example}
\usepackage{xcolor, colortbl}
\definecolor{Gray}{gray}{0.85}

\def\Q{\mathbb{Q}}
\def\Z{\mathbb{Z}}
\def\N{\mathbb{N}}
\def\R{\mathbb{R}}
\def\C{\mathbb{C}}

\def\F{\mathcal{F}}
\def\S{\mathcal{S}}
\def\D{\mathcal{D}}
\def\B{\mathcal{B}}
\def\BB{\boldsymbol{\B}}

\def\uu{\mathbf{u}}
\def\aa{\mathbf{a}}
\def\bb{\mathbf{b}}
\def\xx{\mathbf{x}}
\def\yy{\mathbf{y}}
\def\zz{\mathbf{z}}
\def\t{\mathbf{t}}

\DeclareMathOperator{\Fin}{Fin}
\DeclareMathOperator{\val}{val}
\DeclareMathOperator{\supp}{supp}

\newcommand{\floor}[1]{\lfloor #1 \rfloor}
\newcommand{\ceil}[1]{\lceil #1 \rceil}
\newcommand{\abs}[1]{\lvert #1 \rvert}

\begin{document}

\title{Finiteness property in Cantor real numeration systems}
\author{Zuzana Masáková, Edita Pelantová and Katarína Studeničová\\
FNSPE Czech Technical University in Prague \\
\tt{edita.pelantova@fjfi.cvut.cz}}

\maketitle

\begin{abstract}
We consider a numeration system which is a common generalization of the positional systems introduced by Cantor and Rényi. Number  representations are obtained using a composition of $\beta_k$-transformations for a given  sequence of real bases $\BB=(\beta_k)_{k\geq 1}$, $\beta_k>1$. 
We focus on~arithmetical properties of the set of numbers with finite $\BB$-expansion in case that $\BB$ is an alternate base, i.e.\ $\BB$ is a periodic sequence. We provide necessary conditions for the so-called finiteness property. We further show a~sufficient condition using rewriting rules on the~language of~representations. 
The proof is constructive and provides a~method for~performing addition of~expansions in alternate  bases. Finally, we give a family of alternate bases that satisfy this sufficient condition. Our work  generalizes the results of Frougny and Solomyak obtained for the case when the base $\BB$ is a constant sequence.
\end{abstract}



\section{Introduction }\label{sec:intro}


In the well-known Rényi numeration systems with a real base $\beta>1$, a real number $x\in[0,1)$ is represented by a series $x=\sum_{i\geq 1} x_i\beta^{-i}$, where the digits $x_i$ are obtained by the~$\beta$-transformation $T_\beta:[0,1)\to[0,1)$, $T_\beta(x)=\beta x - \floor{\beta x}$, setting $x_i:=\lfloor\beta T^{i-1}(x) \rfloor\in\{a\in\Z:0\leq a<\beta\}$. The resulting representation of $x$ can be equivalently obtained by the greedy algorithm and is called the (greedy) $\beta$-expansion.
This numeration system was defined and studied from the dynamical point of view in~\cite{Renyi57}. Since then, dozens of authors have studied Rényi numeration systems and their diverse generalizations for their algebraic, geometric or algorithmic features, see for example~\cite{FrSa10}, or~\cite{Rigo2} and the extended bibliography therein.

In this article we consider Cantor real numeration systems as presented in~\cite{CC21}. Similar definition appeared independently in a more general setting in~\cite{CD20}. Note that the concept was proposed already by Galambos~\cite{Galambos}. It is a generalisation of the positional number system considered by Cantor~\cite{Cantor} 
with the aim of providing irrationality criteria for certain class of numbers. 
A Cantor real base $\BB=(\beta_k)_{k\geq 1}$ is a sequence of real numbers $\beta_k>1$ with the condition $\prod_{k\geq 1}\beta_k = +\infty$. Each member of the base sequence has an associated transformation $T_{\beta_i}$ on the interval $[0,1)$.
A real number $x\in[0,1)$ is represented by a series 
$$
x=\sum_{i\geq 1} \frac{x_i}{\beta_1\beta_2\cdots\beta_i}.
$$ 
The integer digits $x_i$ are obtained by applying a composition of $\beta_i$-transformations and they take value in a set depending on the position. In particular 
$$
x_i:=\lfloor\beta_i T_{\beta_{i-1}}\circ \cdots \circ T_{\beta_2}\circ T_{\beta_1}(x) \rfloor \in \{ a\in\Z : 0\leq a <\beta_i \}.
$$
Many new interesting phenomena appear in such numeration systems. A number of results has been provided for the particular situation when the base sequence $\BB$ is purely periodic with period $p$. In this case we speak about alternate base systems and write $\BB=(\beta_1,\dots,\beta_p)$. Note that if the period-length is equal to 1, i.e.\ the base sequence is constant, we obtain the original Rényi numeration system. 

Charlier and Cisternino in~\cite{CC21} study the primary question, namely description of the~Cantor real base representations using the greedy algorithm and characterisation of the~language of greedy representations. They define a subshift corresponding to an alternate base $\BB$ and characterize alternate bases yielding sofic $\BB$-shift. This happens precisely if the greedy expansion of 1 in base $\BB$ and all its shifts $\BB^{(j)}=(\beta_k)_{k\geq j}$ are eventually periodic.
This may be seen as a generalization of the result by Anne Bertrand-Mathis for Rényi numeration systems. 
%
%
Further dynamical properties of alternate base systems are studied in~\cite{CCD21}. 
In \cite{CCMP22}, algebraic description of  alternate
bases with sofic $\BB$-shift is given, together with the construction of a B\"uchi automaton for computation of the normalisation function.

A particular role is played by alternate bases $\BB=(\beta_1, \dots, \beta_p)$ where $\delta=\prod_{i=1}^p\beta_i$ is a~Pisot number and $\beta_i$ belong to the extension field $\Q(\delta)$ for all $i$. As shown in~\cite{CCMP22}, such bases belong to the class of alternate bases with sofic $\BB$-shift. Moreover, such bases stand out when studying the set of numbers with periodic representations, see~\cite{CCK22}, since they satisfy generalization of the theorem by Schmidt~\cite{S80}.

In this paper we focus on the set ${\rm Fin}(\BB)$ of numbers in $[0,1)$ with finite expansions. In~particular,  we aim to identify alternate bases with the so-called finiteness property denoted by (F), i.e.\ such that the set ${\rm Fin}(\BB)$ is closed under addition and subtraction (provided the result belongs to $[0,1)$). We are also interested in a weaker property, the so-called positive finiteness, denoted by (PF), where we require closedness under addition only.

For the case of period $p=1$, i.e.\ the Rényi expansions, the question of finiteness and positive finiteness was originally raised by Frougny and Solomyak~\cite{FS92}. They have shown that if a real base $\beta>1$ satisfies the finiteness property, then $\beta$ is a Pisot number with no positive conjugates, and the corresponding greedy expansion of 1 is finite. 
These conditions are, however, not sufficient and description of bases satisfying the finiteness property turned out to be quite complicated. So far, no exhaustive algebraic characterization has been given.
Nevertheless, some classes of bases $\beta$ with (F) have been found. In~\cite{FS92}, it is proved that (F) is satisfied by the dominant roots of polynomials $x^d-a_{1}x^{d-1}-a_2x^{d-2}-\cdots-a_{d-1}x-a_d$ with positive integer coefficients satisfying $a_1\geq a_2\geq \cdots \geq a_{d-1}\geq a_d$.
Other sufficient conditions were given by Hollander~\cite{Hollander} and Ambrož et al.\ in~\cite{AFMP03}.
Akiyama~\cite{A98} proved that the finiteness property is equivalent to the fact that 0 is an inner point of the so-called central tile associated to the base $\beta$, as defined by Thurston~\cite{Thurston}.
Akiyama then provided a complete description of cubic Pisot units satisfying (F)~\cite{A00}. 
In~\cite{ABBP05}, the correspondence of numeration systems with base $\beta$ satisfying property (F) and the so-called shift radix systems is explained.

In this paper we focus on alternate bases with general $p\geq1$. First we provide necessary conditions for (PF) and (F), in analogy to those of~\cite{FS92}, see Section~\ref{sec:necessary}. Section~\ref{sec:sufficient} presents a sufficient condition for (F) using a system of rewriting rules. The idea stems in rewriting non-greedy $\BB$-representations into lexicographical larger strings representing the same value. Our sufficient condition guarantees that the lexicographically maximal (i.e.\ greedy) $\BB$-representation is obtained in finitely many steps. The most important contribution of this work is given in Section~\ref{sec:nerovnosti} where we provide a class of alternate bases satisfying property (PF) and (F).   

Let us stress that a part of the results of the present paper was announced without full proof in~\cite{MPS23}.


\section{Preliminaries}\label{sec:preliminaries}

\subsection{Finiteness in Cantor real bases}

A \textit{Cantor real base} is a sequence $\BB=  (\beta_n)_{n\geq1}$ of real numbers $\beta_n>1$ such that $\prod_{n\geq1}\beta_n=+\infty$. 
A~\textit{$\BB$-representation} of a number $x\in[0,1]$ is a sequence $\xx=x_1x_2x_3\dots$ of non-negative integers such that 
$$
x=  \sum_{n\geq1}\frac{x_n}{\beta_{1}\cdots \beta_n} =: \val {\xx}.
$$
Note that $\val \xx$ can be defined for any (not only non-negative) sequence $\xx$ of integers. 
We say that the~representation is finite, 
if its support $\supp{\xx}=\{n\in\N,\,n\geq 1 : x_n\neq 0\}$ is finite. 
Denote 
\begin{equation}\label{eq:stringsF}
\F = \{\xx=(x_n)_{n\geq1}\,:\,\supp{\xx}\text{ is finite and  }
 x_n \in \N \ \text{for each }n\}    
\end{equation}
the set of sequences of non-negative integers with finite support. 

A particular $\BB$-representation can be obtained by the \textit{greedy algorithm}:

\noindent For a given $x\in[0,1)$, set $x_1=\lfloor \beta_1 x \rfloor$, $r_1=\beta_1x-x_1$, and for $n\geq 2$ set $x_n=\lfloor \beta_nr_{n-1}\rfloor$, and ${r_n=\beta_nr_{n-1}-x_n}$.  
The obtained $\BB$-representation is denoted by $d_{\BB}(x)=x_1x_2x_3\dots$ and called the \textit{$\BB$-expansion of $x$}. 
It~follows from the greedy algorithm that the digits in the  $\BB$-expansion satisfy $0\leq x_i<\beta_i$ and that $d_{\BB}(x)$ is lexicographically greatest among all  $\BB$-representations of $x$. 
Note that the algorithm works also for~$x=1$; the sequence $d_{\BB}(1)$
is called the \textit{$\BB$-expansion of 1} and it is lexicographically greatest among all $\BB$-representations of 1.

If a $\BB$-expansion is finite, i.e.\ it is of the form $w_1w_2\cdots w_k0^\omega$, where $0^\omega$ stands for infinite repetition of zeros, we sometimes omit the suffix $0^\omega$. We define
\begin{equation}\label{eq:Fin}
\Fin(\BB)=\left\{x\in[0,1) \, : \, d_{\BB}(x)\in\F\right\}.
\end{equation}

\begin{definition}
We say that a Cantor real base $\BB$ satisfies \textit{Property (PF) (the positive finiteness property)}, if for every $x,y\in\Fin(\BB)$ we have
\begin{equation}\label{eq:PF} 
x+y\in[0,1)\implies x+y\in\Fin(\BB).
\end{equation}
We say that a Cantor real base $\BB$ satisfies \textit{Property (F) (the finiteness property)}, if for every $x,y\in\Fin(\BB)$ we have
\begin{equation}\label{eq:F} 
x+y\in[0,1)\implies x+y\in\Fin(\BB)\quad \text{and}\quad 
x-y\in[0,1)\implies x-y\in\Fin(\BB).
\end{equation}
\end{definition}

Property (F) expresses the fact that the set of finite expansions is closed under addition and subtraction of its elements, provided the sum belongs to the interval $[0,1)$. 
Property (PF) focuses only on~addition.

For a given sequence of real numbers $\beta_n>1$, $n\geq 1$, denote by
$\BB^{(k)}$ the Cantor real base ${\BB^{(k)}=(\beta_n)_{n\geq k}}$. 
Realize the following relation between these bases.

\begin{lemma}\label{l:Fsoupnute}
Let $\BB=\BB^{(1)}=(\beta_n)_{n\geq 1}$ be a Cantor real base. If $\BB^{(1)}$ satisfies Property (PF) or Property~(F), then, for every $k\geq 1$, the base $\BB^{(k)}$ satisfies Property (PF) or Property (F), respectively.  
\end{lemma}

\begin{proof}
Consider $z \in [0,1)$. From the greedy algorithm, it is not difficult to show that
\begin{equation}\label{eq:posunrozvoje}
d_{\BB^{(2)}}(z) =  z_{2}z_3z_4\cdots \iff
d_{\BB^{(1)}}(z/\beta_1) =  0z_{2}z_3z_4\cdots .
\end{equation}
Consequently, 
\begin{equation*}
\frac{1}{\beta_1}\Fin\big(\BB^{(2)}\big) \subset \Fin\big(\BB^{(1)}\big).    
\end{equation*}
Take $x,y\in\Fin\big(\BB^{(2)}\big)$ such that $x+y<1$. Then $\frac{x}{\beta_1}$, $\frac{y}{\beta_1}\in \Fin\big(\BB^{(1)}\big)$, and, moreover, 
${z=\frac{x}{\beta_1}+\frac{y}{\beta_1} <\frac1\beta_1<1}$. By Property (PF) of the base $\BB^{(1)}$, necessarily $z\in\Fin(\BB^{(1)})$. 
Moreover, since $z<\frac1\beta_1$, the greedy algorithm implies that the $\BB^{(1)}$-expansion of $z$ is of the~form $d_{\BB^{(1)}}(z) =  0z_{2}z_3z_4\cdots\in\F$. 
Therefore  $d_{\BB^{(2)}}(\beta_1z) =  z_{2}z_3z_4\cdots\in\F$, which means that $x+y\in\Fin\big(\BB^{(2)}\big)$. This proves that $\BB^{(2)}$  satisfies Property~(PF). The proof for subtraction is analogous.
For $\BB^{(k)}$, $k\geq 3$, we proceed by induction.
\end{proof}

\subsection{Admissibility}

Let $\BB=(\beta_n)_{n\geq 1}$ be a Cantor real base.
We say that a sequence $\xx=(x_n)_{n\geq1}$ of integers is \textit{admissible} in~base~$\BB$, if there exists a real number $x\in[0,1)$ such that $\xx = d_{\BB}(x)$. The admissible sequences have been characterized in~\cite{CC21} in terms of the~quasi-greedy expansions of 1.
The \textit{quasi-greedy expansion of 1} in base $\BB$ is defined as 
$$
d_{\BB}^*(1)=\lim_{x\to1-} d_{\BB}(x),
$$
where the limit is taken over the product topology. Note that the quasi-greedy expansion of 1 is lexicographically greatest among all ${\BB}$-representations of 1 with infinitely many non-zero digits.
We obviously have
$$
d_{\BB}^*(1)\preceq_{lex} d_{\BB}(1),
$$
with equality precisely if $d_{\BB}(1)\notin\F$.

\begin{theorem}[\!\cite{CC21}]\label{t:ParryPodminka}
Let  ${\BB}=(\beta_n)_{n\geq 1}$ be a Cantor real base. A sequence of non-negative integers $z_1z_2z_3\cdots$ is admissible in base ${\BB}$ if and only if for each $n\geq 1$
the following inequality holds true
$$
z_{n} z_{n+1}z_{n+2}\cdots \prec_{lex} d_{\BB^{(n)}}^*(1).
$$
\end{theorem}

For recognizing admissibility of finite digit strings, we have a condition formulated using the $\BB$-expansions of 1, instead of the quasi-greedy expansions.

\begin{proposition}\label{p:admisibilitaF}
Let  ${\BB}=(\beta_n)_{n\geq 1}$ be a Cantor real base. A sequence of non-negative integers ${z_1z_2z_3\cdots\in\F}$ is admissible in base $\BB$ if and only if for each $n\geq 1$
the following inequality holds true
$$
z_{n} z_{n+1}z_{n+2}\cdots \prec_{lex} d_{\BB^{(n)}}(1).
$$
\end{proposition} 
\begin{proof}
The implication $\Rightarrow$ follows from Theorem~\ref{t:ParryPodminka}, since $z_{n} z_{n+1}z_{n+2}\cdots \prec_{lex} d_{\BB^*{(n)}}(1)\preceq_{lex} d_{\BB{(n)}}(1)$.
For the opposite implication $\Leftarrow$, we will prove an equivalent statement:

{\it Let  $\zz = z_1z_2z_3 \cdots$ belong to $\F$. If $\zz$ is not  admissible in base $\BB$, then there exists index $\ell\in \N, \ell\geq1$, such that $d_{\BB^{(\ell)}}(1)\preceq z_{\ell}z_{\ell+1}z_{\ell+2}\dots$.} 

By Theorem~\ref{t:ParryPodminka}, if $\zz = z_1z_2z_3 \cdots$ is not admissible in $\BB$, then there exists an index $i\geq 1$ such that 
$d^*_{\BB^{(i)}}(1) \preceq_{lex} z_{i}z_{i+1}z_{i+2}\cdots$. 
Obviously, $i \leq  \max \supp{\zz}$. 
Since $\zz \in \F$ and $d^*_{\BB^{(i)}}(1)\notin\F$, we have a strict inequality 
\begin{equation}\label{eq:dokonceOstra}
  d^*_{\BB^{(i)}}(1) \prec_{lex} z_{i}z_{i+1}z_{i+2}\cdots. 
\end{equation}
If $d_{\BB^{(i)}}(1)\notin\F$, then $d^*_{\BB^{(i)}}(1)=d_{\BB^{(i)}}(1)$ and the proof is finished. Consider the other case, namely that $d_{\BB^{(i)}}(1)=w_1w_2\cdots w_{n-1}w_n 0^\omega\in\F$, where  $w_n \geq 1$.  It has been shown in~\cite{CC21} that 
 $d^*_{\BB^{(i)}}(1) = w_1\cdots w_{n-1}(w_n-1)d^*_{\BB^{(i+n)}}(1)$. Inequality \eqref{eq:dokonceOstra} implies 
 \begin{enumerate}
     \item[\itshape 1.] either $w_1\cdots w_{n-1}(w_n-1) \prec_{lex} z_{i}z_{i+1}\cdots z_{i+n-1} $, 
     \item[\itshape 2.]  or $d^*_{\BB^{(i+n)}}(1) \prec_{lex} z_{i+n}z_{i+n+1}\cdots $.
 \end{enumerate}

\textit{Case 1.}  
As the lexicographically smallest word of length $n$ strictly  greater than  $w_1\dots w_{n-1}(w_n-1) $ is $w_1\dots w_{n-1}w_n $,  we have  $w_1\dots w_{n-1}w_n \preceq_{lex}z_{i}z_{i+1} \cdots z_{i+n-1}$. Thus
$$
d_{\BB^{(i)}}(1)\preceq_{lex} z_{i}z_{i+1} \dots z_{i+n-1} 0^\omega \preceq_{lex} z_{i}z_{i+1}z_{i+2}\dots, 
$$ 
as we wanted to show. 
 
\textit{Case 2.}  
We have found $i_{new}= i+n>i$ so that~\eqref{eq:dokonceOstra} is satisfied when substituting  $i_{new}$ in place of $i$. The same idea can be repeated. As $\supp{\zz}$ is bounded, finitely many steps lead to Case 1.   
\end{proof}

\subsection{Alternate bases}

In case that the Cantor real base $\BB=(\beta_n)_{n\geq 1}$ is a purely periodic sequence with period $p$, we call it an~\textit{alternate base} and denote
$$
\BB=(\beta_1, \beta_2, \dots,\beta_p),\qquad \delta=\prod_{i=1}^p\beta_i.
$$ 
For the Cantor real bases 
$\BB^{(n)}=(\beta_{n},\beta_{n+1},\dots,\beta_{n+p-1})$, 
we have $\BB^{(n)}=\BB^{(n+p)}$ for every $n\geq 1$. The~indices will be therefore considered $\!\!\!\!\mod p$. Let us stress that throughout this paper the representatives of~the~congruence classes $\!\!\!\!\mod p$ will be taken in the set $\Z_p=\{1,2,\dots,p\}$. 

For the Cantor real base 
$\BB^{(n)}$, $n\geq 1$, denote the $\BB^{(n)}$-expansions of 1 by
\begin{equation}\label{eq:t}
  \t^{(n)}=t_1^{(n)}t_2^{(n)} t_3^{(n)}\cdots = d_{\BB^{(n)}}(1).
\end{equation} 
Note that $\t^{(n)}=\t^{(n+p)}$ for each $n\in\N$. In fact, we only work with $\t^{(n)}$, $n\in\Z_p$.

The specific case when $p=1$ yields the well-known numeration systems in one real base $\beta>1$ as~defined by Rényi~\cite{Renyi57}.  
These systems  define a sofic $\beta$-shift precisely when the base $\beta$ is a Parry number. 
In analogy to this case we say that the alternate base $\BB$ with period $p\geq 1$ is a \textit{Parry alternate base}, if
all the sequences $\t^{(\ell)}$, $\ell\in\Z_p$, are eventually periodic.
If, moreover, $\t^{(\ell)}\in\F$ for all $\ell\in\Z_p$, we speak about \textit{simple Parry alternate base}.

Note that~\eqref{eq:posunrozvoje} for an alternate base $\BB$ implies
\begin{equation}\label{eq:posunrozvojealternate}
d_{\BB}(z) =  z_{1}z_{2}z_{3}\cdots \iff
d_{\BB}(z/\delta) =  0^pz_{1}z_2z_2\cdots 
\end{equation}
and 
$\frac{1}{\delta}\Fin(\BB)\subset\Fin(\BB)$.
We thus have
a simple consequence of Lemma~\ref{l:Fsoupnute}.

\begin{corollary}\label{c:vsechnyF}
Let $\BB=(\beta_1,\beta_2,\dots,\beta_p)$ be an alternate base.  
Then  for any $\ell\in\Z_p$ the~base
$\BB^{(\ell)}$ satisfies Property (PF) or Property (F), if and only if
$\BB$ satisfies Property (PF) or Property (F), respectively.
\end{corollary}

\section{Main results}\label{sec:mainresults}

We focus on the finiteness property of alternate bases.
We first provide necessary conditions on an~alternate base in order to satisfy Property (PF) and Property (F).
For that, we need to recall some number theoretical notions. 
An algebraic integer $\delta>1$ is a~\textit{Pisot number}, if all its conjugates other than $\delta$ belong to the interior of the unit disc; an algebraic integer $\delta>1$ is a~\textit{Salem number}, if all its~conjugates other than $\delta$ belong to the unit disc and at least one lies on its boundary. If $\delta$ is an algebraic number of~degree~$d$, then the minimal subfield of~$\C$ containing $\delta$ is of the form
$$
\Q(\delta)=\left\{c_0+c_1\delta+\cdots +c_{d-1}\delta^{d-1} \, : \, c_i\in\Q\right\}.
$$
The number field $\Q(\delta)$ has precisely $d$ embeddings into $\C$, namely the field isomorphisms ${\psi:\Q(\delta)\to\Q(\gamma)}$ induced by $\psi(\delta)=\gamma$, where $\gamma$ is a conjugate of $\delta$. Note that one of the~embeddings is the identity map.

The necessary condition is formulated as Theorem~\ref{thm:necessary}. We provide its proof in Section~\ref{sec:necessary} divided into several propositions.

\begin{theorem}\label{thm:necessary}
    Let $\BB=({\beta_1,\beta_2,\dots,\beta_p})$ be an alternate base. 
    \begin{enumerate}
    \item
    If $\BB$ satisfies Property (PF), then 
    $\delta=\prod_{i=1}^p\beta_i$ is a Pisot or a Salem number and 
    $\beta_i\in\Q(\delta)$ for~${i \in \Z_p}$.
    \item 
    If, moreover, $\BB$ satisfies Property (F), then 
    $\BB$ is a simple Parry alternate base, and
    for any non-identical embedding $\psi$ of $\Q(\delta)$ into $\C$ the vector $(\psi(\beta_1),\dots,\psi(\beta_p))$ is not positive.
    \end{enumerate}
    \end{theorem}

As the main result of this paper we provide a class of alternate Parry bases 
satisfying the finiteness property. It is a generalization of the class given in \cite{FS92}. Our proof follows similar ideas as in \cite{AFMP03}, but requires much deeper techniques.

\begin{theorem}\label{thm:nerovnosti}
Let $\BB$ be an alternate base with period $p$ such that the corresponding expansions $\t^{(\ell)}$ of 1, as defined in~\eqref{eq:t}, satisfy
\begin{equation}\label{eq:GFS} 
t_1^{(\ell)}\geq  \ t_2^{(\ell - 1)}\geq \   t_3^{(\ell-2)}\geq \  \cdots \quad \text{
    for every } \ \ell \in \mathbb{Z}_p.
\end{equation} 
Then $\BB$ is a Parry  alternate base and  has property (PF). 

If, moreover, $\BB$ is a simple Parry alternate base, then it satisfies Property (F).
\end{theorem}

Note that in inequalities~\eqref{eq:GFS} we have upper indices in $\Z$, but they are always counted $\!\!\!\mod p$ and taking values in the set $\{1,2,\dots,p\}$.

The statement of Theorem~\ref{thm:nerovnosti} is obtained as a consequence of a sufficient condition for (PF) given in Section~\ref{sec:sufficient}. The sufficient condition is given in terms of rewriting rules of~non-admissible sequences (Theorem~\ref{thm:sufficient}). 
Then, in Section~\ref{sec:nerovnosti}, we prove that the class of bases satisfying the~inequalities~\eqref{eq:GFS} meets the assumptions of Theorem~\ref{thm:sufficient}, and thus possesses Property (PF).

Let us illustrate the above results on the favourite example of Charlier and Cisternino, given in~\cite{CC21}.

\begin{example}\label{ex:1}
    Consider 
    $\BB=
    (\beta_1,\beta_2)=
    \left(\frac{1+\sqrt{13}}2, \frac{5+\sqrt{13}}6\right)$. Expansions of 1 are of the form 
    $$
    d_{\BB^{(1)}}(1)=t_1^{(1)}t_2^{(1)}t_3^{(1)}0^\omega=2010^\omega \quad \text{ and }\quad d_{\BB^{(2)}}(1)=t_1^{(2)}t_2^{(2)}0^\omega=110^\omega.
    $$ 
    Consequently, $\BB$ is a simple Parry alternate base. The inequalities \eqref{eq:GFS} for $p=2$ are of the~form
    \begin{align*}
        &t_1^{(1)} \geq t_2^{(2)} \geq t^{(1)}_{3}\geq t_4^{(2)} \geq\cdots, &&\text{ in our case } \hspace{3em} 2 \geq 1 \geq 1\geq 0  \geq \cdots,
        \\  &t_1^{(2)} \geq t_2^{(1)} \geq t^{(2)}_{3}\geq t_4^{(1)}\geq \cdots, &&\text{ in our case } \hspace{3em} 1 \geq 0 \geq 0 \geq 0 \geq \cdots.
    \end{align*}
    Therefore, according to Theorem \ref{thm:nerovnosti}, the base $\BB$ has Property (F). 
    
    Note that the necessary conditions of Property (F) as presented in Theorem~\ref{thm:necessary} are satisfied. Indeed, 
    we have $\delta = \frac{1+\sqrt{13}}2\cdot\frac{5+\sqrt{13}}6=\frac{3+\sqrt{13}}2$, which is the positive root of $x^2-3x-1$. The other root of this polynomial is $\gamma =\frac{3-\sqrt{13}}2 \approx -0.303$, thus $\delta$ is a Pisot number.

    Obviously, $\beta_1,\beta_2 \in \Q(\delta)=\Q(\sqrt{13})$. The only non-identical embedding $\psi$ of the field $\Q(\sqrt{13})$ into $\C$ is the Galois automorphism induced by $\sqrt{13} \mapsto -\sqrt{13}$. Thus we have
    \begin{equation*}
    \psi(\beta_1)=\frac{1-\sqrt{13}}2 \approx -1.303, \hspace{5em}
    \psi(\beta_2)=\frac{5-\sqrt{13}}6 \approx 0.232,
    \end{equation*}
    therefore indeed $\left(\psi(\beta_1),\psi(\beta_2)\right)$ is not a positive vector.  
\end{example}

\section{Necessary conditions}\label{sec:necessary}

The proof of Theorem~\ref{thm:necessary} is divided into four propositions.  
Throughout this section we have an alternate base $\BB=({\beta_1,\beta_2,\dots,\beta_p})$  and 
 $\delta=\prod_{i=1}^p \beta_i$. In the demonstration, we will use that 
 $\prod_{i=1}^{mp}\beta_i=\delta^{m}$, which gives 
 a simple 
 expression that holds for any $n\in\N$ and any sequence of digits $(c_k)_{k=1}^{pn}$, namely
\begin{equation}\label{eq:uprava}
\sum_{k=1}^{pn} \frac{c_k}{\prod_{i=1}^k \beta_i}  = \frac{1}{\delta^n} \sum_{j=1}^p\ \Bigg(\prod_{i=j+1}^p \beta_i\Bigg) \sum_{k=0}^{n-1} \delta^{n-1-k}c_{kp+j}.
\end{equation}

The first step is a rather technical lemma necessary for the proof of Proposition~\ref{prop:nec_F_4}.

\begin{lemma}\label{l:polynomy} Let  $\BB=({\beta_1,\beta_2,\dots,\beta_p})$  be an alternate base satisfying Property (PF). Then for any sufficiently large $n \in \N$ there exist  polynomials $g_1, g_2, \ldots , g_p \in \Z[X] $ such that 

\begin{itemize}
\item the degree of $g_j$ is at most $n-1$ for every $j \in \{1,2,\ldots,  p-1\}$;
\item $g_p$ is monic  and  its degree is $n$;
\item all coefficients of $g_j$ are non-negative for every $j \in\{2,\ldots,p\}$;  
\item  the product of the  row vector $\bigl(g_1(\delta), \ldots, g_p(\delta)\bigr)$  and the column vector 
$\vec{v} =\! \bigg(\!\prod\limits_{i=2}^p \beta_i, \prod\limits_{i=3}^p \beta_i,
\ldots,  \prod\limits_{i=p}^p \beta_i, 1\bigg)^T $ is zero. 
\end{itemize}
  \end{lemma}

\begin{proof} 
Choose $m \in \N$ such that  $\delta^m \beta_1 \leq \ceil{\delta^m \beta_1} < \delta^m \beta_p\beta_1$. 
Hence $z :=  {\ceil{\delta^m \beta_1}}({\delta^m \beta_p\beta_1})^{-1} $ satisfies inequalities  $\frac{1}{\beta_p}\leq z <1$.  
It is easy to see that $({\delta^m \beta_p\beta_1})^{-1}\in  \Fin\big(\BB^{(p)}\big)$, since its $\BB^{(p)}$-expansion is of~the~form $0^{mp+1}10^\omega$. 
By Corollary \ref{c:vsechnyF}, Property (PF) of the base $\BB$ implies that also the base ${\BB^{(p)} = (\beta_p, \beta_{1}, \ldots, \beta_{p-1})}$ has Property~(PF). 
Thus any integer multiple of $({\delta^m \beta_p\beta_1})^{-1}$, in particular the~number $z$, has a finite $\BB^{(p)}$-expansion. 
We easily see that 
$$
{\frac{1}{\beta_p}\leq z=\frac{\ceil{\delta^m \beta_1}}{{\delta^m \beta_p\beta_1}}< \frac{\delta^m \beta_1+1}{\delta^m \beta_p\beta_1}=\frac{1}{\beta_p} + \frac{1}{\delta^m\beta_p\beta_1}<\frac2{\beta_p}},
$$ 
and~thus the greedy algorithm 
returns the $\BB^{(p)}$-expansion of $z$ of the form ${d_{\BB^{(p)}}(z) = 1z_1z_2\cdots z_{pn} 0^\omega}$, $z_i \in \N$, for sufficiently large $n\in \N, n> m$. Using~\eqref{eq:uprava}, we derive 
$$
\frac{\ceil{\delta^m \beta_1}}{\delta^m \beta_p\beta_1} =  \frac{1}{\beta_p} + \frac{1}{\beta_p} \sum_{k=1}^{pn} \frac{z_k}{\prod_{i=1}^k \beta_i}  = \frac{1}{\beta_p} + \frac{1}{\delta^n\beta_p} \sum_{j=1}^p\ \Bigg(\prod_{i=j+1}^p \beta_i\Bigg) \sum_{k=0}^{n-1} \delta^{n-1-k}z_{kp+j}.
$$
In order to simplify notation, denote the $j$-th component of the vector  $\vec{v}$  by $v_j = \prod_{i=j+1}^p\beta_i$. 
Multiplying the above equality by $\delta^n\beta_p$ yields
\begin{equation}\label{eq:skorohotove}
\ceil{\delta^m \beta_1}\delta^{n-m-1}v_1 =  \delta^nv_p + \sum_{j=1}^p\ v_j\sum_{k=0}^{n-1} \delta^{n-1-k}z_{kp+j}. 
\end{equation}
Now we can define polynomial $g_j$, $j\in\Z_p$, as follows.
$$
\begin{aligned}
g_1(X) &:= -\ceil{\delta^m \beta_1}X^{n-m-1} + \sum_{k=0}^{n-1} X^{n-k-1}z_{kp+1},\\ 
g_j(X) &:=  \sum_{k=0}^{n-1} X^{n-k-1}z_{kp+j}\quad\text{ for }j \in\{2,3,\ldots, p-1\},\\ 
g_p(X) &:= X^n + \sum_{k=0}^{n-1} X^{n-k-1}z_{kp+p}. 
\end{aligned}
$$
With this notation, equality~\eqref{eq:skorohotove} is of the form $0 = g_1(\delta)v_1 + g_2(\delta)v_2+\cdots + g_p(\delta)v_p$. One easily verifies all required properties of polynomials $g_j$, $j\in\Z_p$.
\end{proof}

The proof of the  following proposition is  a modification of  the proof of Theorem~14 in~\cite{CCMP22}. 

\begin{proposition}
\label{prop:nec_F_4}
Let $\BB=({\beta_1,\beta_2,\dots,\beta_p})$ be an alternate base 
with Property (PF). Then $\delta=\prod_{i=1}^p \beta_i$ is an algebraic integer and $\beta_i\in\Q(\delta)$ for all $i\in\Z_p$. 
\end{proposition}

\begin{proof} 
By Corollary \ref{c:vsechnyF},  $\BB^{(\ell)}$ satisfies Property (PF)   for every $\ell \in \Z_p$. Choose $n\in \N$ sufficiently large so that Lemma \ref{l:polynomy} can be applied for every $\ell \in \Z_p$ with the same value of $n$. By that, we obtain
polynomials $g^{(\ell)}_1, g^{(\ell)}_2, \ldots , g^{(\ell)}_{p-1}, g^{(\ell)}_{p}  \in \Z[X] $ with properties listed in the lemma; in particular, the vectors $\Big(g^{(\ell)}_1(\delta),g^{(\ell)}_2(\delta) , \ldots, g^{(\ell)}_p(\delta)\Big)$ and  $\vec{v}^{(\ell)} = \bigg(\prod\limits_{i=2}^p \beta_{i+\ell -1}, \, \prod\limits_{i=3}^p \beta_{i+\ell -1},\, 
\ldots,  \prod\limits_{i=p}^p \beta_{i+\ell -1},\, 1\bigg)^T $ 
satisfy
\begin{equation}\label{eq:rovnice}
    \Big(g^{(\ell)}_1(\delta), \ldots, g^{(\ell)}_p(\delta)\Big) \, \vec{v}^{(\ell)} = 0 \quad \text{for every } \ell \in\Z_p. 
\end{equation} 
Now consider $\ell\in\Z_p$ fixed. For $k\in\Z_p$, the $k$-th component of the vector $\vec{v}^{(\ell)}$ is equal to 
$$
{v}^{(\ell)}_k=\prod_{i=k+1}^{p}\beta_{i+\ell-1} = \prod_{j=k+\ell}^{p+\ell-1}\beta_j = 
    \begin{cases}
     \prod_{j=k+\ell-p}^{\ell-1}\beta_j& \text{ for }k+\ell\geq p+2;\\[2mm]
     \Big(\prod_{j=k+\ell}^{p}\beta_j\Big)\cdot \Big(\prod_{j=1}^{\ell-1}\beta_j\Big) & \text{ for }k+\ell\leq p+1.
    \end{cases}
$$
As a result, multiplying the vector $\vec{v}^{(\ell)}$ by the constant $\prod_{i=\ell}^p\beta_i$ leads to
$$
\Bigg(\prod_{i=\ell}^p\beta_i\Bigg) \ \vec{v}^{(\ell)} = 
\Big(\delta {v}^{(1)}_{\ell}, \dots, \delta {v}^{(1)}_p, 
{v}^{(1)}_1,\dots,{v}^{(1)}_{\ell-1}
\Big)^T= 
    \begin{pmatrix} 
    {\Theta} & \delta I_{p-\ell+1}\\
    I_{\ell-1} & \Theta
    \end{pmatrix}
\vec{v}^{(1)},
$$
where $I_j$ denotes the identity matrix of order $j$ and $\Theta$ stands for rectangular zero matrix of suitable size. 
We have thus derived that for $\ell\in\Z_p$ 
\begin{equation}\label{eq:rovnice1}
\Bigg(\prod_{i=\ell}^p \beta_i\Bigg)\  \vec{v}^{(\ell)} = R^{p-\ell+1}\  \vec{v}^{(1)}\quad \text{where} \quad 
R=
    \begin{pmatrix} 
    {\Theta} & \delta I_1 \\
    I_{p-1} & \Theta
    \end{pmatrix}.
\end{equation}
Combining~\eqref{eq:rovnice} with~\eqref{eq:rovnice1}, we obtain that for $\ell\in\Z_p$
$$  
\Big(g^{(\ell)}_1(\delta), \ldots, g^{(\ell)}_p(\delta)\Big)\ R^{p-\ell+1} \, \vec{v}^{(1)} = 0 .
$$
In other words, the vector $\vec{v}^{(1)}$  satisfies 
$M\vec{v}^{(1)} = \vec{0}$, where the matrix $M$ is a square matrix of~order~$p$ of~the~form
$$
M =\begin{pmatrix}
\delta g_{1}^{(1)}&\delta g_{2}^{(1)}&\delta g_{3}^{(1)}& \cdots & \delta g_{p-1}^{(1)}& \delta g_{p}^{(1)}\\
g_{p}^{(2)}&\delta g_{1}^{(2)}&\delta g_{2}^{(2)}& \cdots &\delta g_{p-2}^{(2)}& \delta g_{p-1}^{(2)}\\
 g_{p-1}^{(3)} &g_{p}^{(3)}&\delta g_{1}^{(3)}& \cdots &\delta g_{p-3}^{(3)}&\delta g_{p-2}^{(3)}\\

\vdots &\vdots&\vdots& \ddots  & \vdots& \vdots \\
 g_{3}^{(p-1)} &g_{4}^{(p-1)}& g_{5}^{(p-1)}& \cdots &\delta g_{1}^{(p-1)}&\delta g_{2}^{(p-1)}\\
  g_{2}^{(p)} &g_{3}^{(p)}&g_{4}^{(p)}& \cdots & g_{p}^{(p)}&\delta g_{1}^{(p)}\\
  \end{pmatrix}.
$$
Note that the  components of the matrix $M$ are polynomials in $\delta$ 
(for simplicity of notation, we have omitted the dependence on the variable). Their degree is given by Lemma~\ref{l:polynomy}. 
In particular, the element of $M$ with strictly highest degree is $M_{1,p}=\delta g_p^{(1)}$ of degree $n+1$.  

Equality $M\vec{v}^{(1)} = \vec{0}$ means that $\vec{v}^{(1)}$ is an eigenvector of the matrix $M$ corresponding to the eigenvalue~$0$. Since $\vec{v}^{(1)}\neq \vec{0}$, necessarily $\det M =0$. 
We now show that $\det M$ is a monic polynomial with integer coefficients in the variable $\delta$. Consequently, $\delta$ is an algebraic integer.
By Lemma~\ref{l:polynomy}, the product
$$
M_{1, p}M_{p, p-1}M_{p-1, p-2} \cdots M_{3,2} M_{2,1} = \delta \prod_{\ell=1}^p g_p^{(\ell)}(\delta)
$$ 
is a product of  monic polynomials, together of degree $pn+1$. 
All other products contributing to the~value of~$\det M$ are of degree~$\leq pn$. Hence, $\det M$ is a monic polynomial in $\delta$ of degree $pn+1$ with integer coefficients.  

It remains to show that  $\beta_i \in \Q(\delta)$ for every $i\in\Z_p$. Recall that  $\vec{v}^{(1)}$ is an eigenvector of the matrix $M$ to the eigenvalue~$0$. 
We now show that $0$ is a simple eigenvalue of $M$. For that, we use Perron-Frobenius theorem. 
Note that all non-diagonal components of $M$ are non-negative, since $\delta>1$ and the~coefficients of the polynomials $g^{(\ell)}_j$ for $j\geq2$ are non-negative. Thus for sufficiently large $c\in \N$, the matrix  $M+cI_p$ is non-negative. The matrix is also irreducible. This is readily seen, realizing that  entries 
$M_{1,p} = \delta g_p^{(1)}(\delta)$, $M_{2,1} =g_p^{(2)}(\delta)$, \ldots, $M_{p, p-1}=g_p^{(p)}(\delta) \geq \delta^n >0$ are positive. 
The vector $\vec{v}^{(1)}$ is a positive eigenvector of the matrix $M+cI_p$ to the eigenvalue $c$, i.e.\ 
$(M+cI_p)\vec{v}^{(1)}=c\vec{v}^{(1)}$. By Perron-Frobenius theorem, $c$ is a~simple eigenvalue of $M+cI_p$ and hence $0$ is a~simple eigenvalue of~$M$.

As components of the matrix $M$ belong to the field $\Q(\delta)$,
there exists a vector $\vec{u}$ with components in~$\Q(\delta)$ such that $M \vec{u} = \vec{0}$. We necessarily have $\vec{v}^{(1)} = \alpha\vec{u}$ for some $\alpha \in \R$. 
To complete the proof, realize that for every $i \in \{2,3, \ldots, p\}$ 
$$
\beta_{i} = \frac{v_{i-1}^{(1)}}{v_{i}^{(1)}} = \frac{u_{i-1}}{u_{i}}  \in~\Q(\delta),
$$
and $\beta_1 = \frac{\delta}{\beta_2\cdots\beta_p} \in \Q(\delta)$.
\end{proof}

\begin{proposition}
\label{prop:nec_F_5}
Let $\BB=({\beta_1,\beta_2,\dots,\beta_p})$ be an alternate base 
with Property (PF). Then $\delta=\prod_{i=1}^p \beta_i$ is either a Pisot or a Salem number.   
\end{proposition}

The proof follows the same ideas as the proof of Theorem~1 in~\cite{CCK22}.

\begin{proof}
According to Proposition \ref{prop:nec_F_4} the number $\delta$ is an algebraic integer. 
If $\delta$ is a rational integer, the~proof is finished. Assume that $\delta \notin \N$.  
For any sufficiently large $m \in \N$,  one has
\begin{equation*}
    \delta^m<\ceil{\delta^m}<\delta^m \beta_p.
\end{equation*}
Since $\frac{1}{\delta^{m+1}} \in \Fin(\BB)$,  Property (PF) implies that $n \,\frac{1}{\delta^{m+1}} \in \Fin(\BB)$  for all $n \in \N, n < \delta^{m+1}$.  
In~particular, $x:= \ceil{\delta^m}\frac{1}{\delta^{m+1}}  \in \Fin(\BB) $ as well, i.e.\, $d_{{\BB}}(x) = x_1x_2x_3\cdots x_{pN}0^\omega$  for some $N \in \N$. 
Realize that $x$ satisfies 
$$
\frac{1}{\delta} < x=\frac{\ceil{\delta^m}}{\delta^{m+1}}<\frac1\delta + \frac{1}{\delta^m} < \frac{2}{\delta},
$$ and therefore the greedy algorithm yields $x_1 = \dots =x_{p-1} = 0$, $x_p = 1$ and $r_p = \frac{\ceil{\delta^m} - \delta^m}{\delta^m} \in \frac{1}{\delta^m}(0,1)$. Hence $x_{p+1}=x_{p+2} = \dots = x_{mp} = 0$. 
When evaluating the string $d_{{\BB}}(x)$, we may group summands in~the~following way
\begin{equation}\label{sgrupovane} 
x= \frac{\ceil{\delta^m}}{\delta^{m+1}}   = \frac{1}{\delta} + \sum_{k=mp+1}^{pN} \frac{x_k}{\prod_{i=1}^k{\beta_i}} = \frac{1}{\delta} + \sum_{j=m+1}^{N}\frac{1}{\delta^{j+1}}\underbrace{\Bigg(\sum_{k=1}^{p} x_{jp+k}\prod_{i=k+1}^p{\beta_i}\Bigg)}_{=: d_j}.
\end{equation}
Greedy algorithm for obtaining $d_{\BB}(x)$ implies that the  digit  $x_i$ is  bounded from  above by $\beta_i \leq \delta$, for each $i \in \N, i\geq 1$, and thus every coefficient $d_j$ belongs to the finite alphabet 
$$
\D=\Bigg\{a_1 \prod_{i=2}^p\beta_i + a_2 \prod_{i=3}^p\beta_i+\cdots + a_{p-1}\prod_{i=p}^p\beta_i + a_p \, : \, a_j \in \N, a_j\leq \floor{\delta}\Bigg\}.
$$
In this notation, equality~\eqref{sgrupovane} multiplied by $\delta^{m+1}$ can be rewritten as 
\begin{equation}\label{sgrupovane2}
\ceil{\delta^m} - \delta^m =\sum_{j=m+1}^{N} \frac{d_{j}}{\delta^{j-m}}. \end{equation}
In order to show that $\delta$ is a Pisot or a Salem number, denote $\gamma$ a conjugate of $\delta$, $\gamma \neq \delta$, and $\psi$ the~isomorphism between  $\Q(\delta)$ and $\Q(\gamma)$ induced by $\psi(\delta) = \gamma$.  
Since $\ceil{\delta^m} \in \N$, we have $\psi(\ceil{\delta^m} )=\ceil{\delta^m}$.
By~Proposition~\ref{prop:nec_F_4}, all $\beta_i$ belong to the field $\Q(\delta)$,
and thus the image of the isomorphism $\psi$ on the~elements of~the~set $\D$ is well defined. Application of $\psi$ on \eqref{sgrupovane2} yields
\begin{equation}
\label{eq_odhad}
\delta^m - \gamma^m = \psi\Bigl(\ceil{\delta^m} - \delta^m\Bigr) - \Bigl(\ceil{\delta^m} - \delta^m\Bigr) =   \sum_{j=m+1}^{N} \frac{\psi(d_j)}{\gamma^{j-m}}  -\sum_{j=m+1}^{N} \frac{d_j}{\delta^{j-m}}.
\end{equation}
Let us show that  the assumption $\abs{\gamma}>1$ leads to a contradiction. 

Denote $M=\max\{ |\psi(d)|\,:\, d \in \D\}$.  Obviously
\begin{equation}
\label{eq_odhad2}
\Biggl|\,\sum_{j=m+1}^{N} \frac{\psi(d_j)}{\gamma^{j-m}}\,\Biggr| \leq M\frac{1}{|\gamma| -1} \qquad 
\text{and}
\qquad \Biggl|\sum_{j=m+1}^{N} \frac{d_j}{\delta^{j-m}}\Biggr| = \Bigl|\ceil{\delta^m}-\delta^m\Bigr|<1.
\end{equation}
Since neither $M$ nor $\gamma$ depends on the choice of $m\in \N$, relations~\eqref{eq_odhad} and~\eqref{eq_odhad2} imply that the sequence ${\bigl(\delta^m - \gamma^m\bigr)_{m \in \N}}$ is bounded. 
This is in contradiction with the fact that $\limsup\limits_{m\to +\infty}|\delta^m - \gamma^m| =+\infty$. 
\end{proof}

Propositions~\ref{prop:nec_F_4} and~\ref{prop:nec_F_5} demonstrate Item 1 of Theorem~\ref{thm:necessary}. Let us proceed with the~proof of Item~2.

\begin{proposition}
\label{l_nec_F_1}
Let $\BB=({\beta_1,\beta_2,\dots,\beta_p})$ be an alternate base. If $\BB$ satisfies (F), then $\BB$ is a simple Parry alternate base.
\end{proposition}

\begin{proof}
Denote $\delta = \prod _{i=1}^p\beta_i$ and $d_{\BB}(1) = t_1t_2t_3 \cdots$. It suffices to show that $d_{\BB}(1) \in \F$. The~fact that $d_{\BB^{(\ell)}}(1) \in \F$ for every $\ell\in\Z_p$ then follows from Corollary \ref{c:vsechnyF}.

Clearly, $\frac{1}{\delta}, \frac{1}{\delta \beta_1}\in \Fin(\BB)$, and  
$x:=\frac{1}{\delta} - \frac{t_1}{\delta \beta_1}=\frac{\beta_1-\floor{\beta_1}}{\delta \beta_1}\in [0,1)$.  
Property (F) implies that $x$ belongs to the set $\Fin(\BB)$ as well.  
Obviously, $x<\frac{1}{\delta\beta_1}$. Therefore, the greedy algorithm for the $\BB$-expansion $d_{\BB}(x)=x_1x_2x_3\cdots$ yields $x_1=x_2=\dots = x_{p+1}=0$ and
$$
r_{p+1} =  \delta\beta_1 x=\beta_1-{t_1}= \frac{t_2}{\beta_2} + \frac{t_3}{\beta_2\beta_3}  +\cdots\ . 
$$
Consequently, $0^{p+1}t_2t_3t_4\cdots$ is the $\BB$-expansion of $x$.
Since $x \in \Fin(\BB)$, the support of $d_{\BB}(x)=0^{p+1}t_{2}t_{3}t_{4}\cdots$ is finite, and thus also the support of $d_{\BB}(1)= t_1t_2t_3\cdots$ is finite.
\end{proof}

The following proposition completes the proof of Theorem~\ref{thm:necessary}.

\begin{proposition}
\label{prop:nec_F_6}
Let $\BB=({\beta_1,\beta_2,\dots,\beta_p})$
be an alternate base with Property (F) and let $\psi$ be a non-identical embedding of~$\Q(\delta)$ into $\C$. Then the~vector $(\psi(\beta_1),\dots,\psi(\beta_p))$ is not positive.
\end{proposition}

\begin{proof}  
By Proposition~\ref{l_nec_F_1}, all $d_{\BB^{(i)}}(1)$ have finite support, thus there exists $n \in \N$ such that we may write
\begin{equation*}
    d_{\BB^{(i)}}(1)=t_1^{(i)}t_2^{(i)}\dots t_n^{(i)}0^\omega
\end{equation*}
for every $i \in \{1,\dots,p\}$. Therefore we have the following relation
\begin{equation}\label{eq:jedna}
    1=\frac{t_1^{(i)}}{\beta_{i}}+\frac{t_2^{(i)}}{\beta_{i} \beta_{i+1}}+\dots+\frac{t^{(i)}_n}{\beta_{i}\cdots \beta_{i+n-1}}.
\end{equation}
for every $i$. We now show that the assumption of positivity of the vector  $(\psi(\beta_1),\dots,\psi(\beta_p))$ leads to a contradiction with 
the fact that, according to Proposition \ref{prop:nec_F_5}, the number $\delta$ is either a Pisot or a Salem number.
Suppose that $\psi(\beta_j)>0$ for all $j \in \{1,\dots,p\}$.
Multiplying~\eqref{eq:jedna} by $\beta_{i}$, subtracting $t_1^{(i)}$ and applying $\psi$ yields
\begin{equation*}
\psi(\beta_{i})-t_1^{(i)}=\frac{t^{(i)}_2}{\psi(\beta_{i+1})}+\frac{t_3^{(i)}}{\psi(\beta_{i+1})\psi(\beta_{i+2})}+\dots+\frac{t_n^{(i)}}{\psi(\beta_{i+1})\cdots \psi(\beta_{i+n-1})} \geq 0.
\end{equation*}
Therefore $\psi(\beta_i)\geq t_1^{(i)}$~for all indices $i \in \{1,\dots,p\}$. Consequently,
\begin{equation*}
    \psi(\delta)=\psi(\beta_1)\cdots\psi(\beta_p)\geq\prod_{i=1}^{p} t_1^{(i)}\geq 1.
\end{equation*}
Since $\delta\neq 1$, we have $\psi(\delta)>1$.
Realize that $\psi(\delta)$ is an algebraic conjugate of $\delta$, thus we derive that $\delta$ is neither a Pisot, nor a Salem number, which contradicts Proposition \ref{prop:nec_F_5}.
\end{proof}

\begin{remark}
In case that $p=1$, the above statement directly implies that $\delta=\beta$ is not a~Salem number, since all Salem numbers have a conjugate in $(0,1)$.
For $p\geq 2$, we were not able to exclude $\delta$ to be a~Salem number. 
\end{remark}

\section{Sufficient condition}\label{sec:sufficient}

In this section we prove a sufficient condition for an alternate base to satisfy Property (PF). For our purposes, Property (PF) can be reformulated in terms of strings in the set $\F$, as defined in~\eqref{eq:stringsF}.
For that, let us introduce the notation for digit-wise addition of strings.
Let $\xx = x_1x_2x_3 \cdots $ and $\yy = y_1y_2y_3 \cdots$ be two sequences with $x_n, y_n \in \mathbb{Z}$ for every $n \in \N, n \geq 1$. Then $\xx\oplus \yy$ stands for the sequence $z_1z_2z_3 \cdots $, where  $z_n =x_n+y_n $ for every $n \in \N$, $n\geq 1$.
Obviously, in this notation, we have
\begin{itemize}
\item
$\val  (\xx\oplus \yy) = \val  \xx + \val \yy$;  

\item if $\xx\prec_{lex} \tilde{\xx}$, then for every $\yy$ the inequality $\xx\oplus\yy \prec_{lex} \tilde{\xx}\oplus\yy$ holds true.
\end{itemize}

\begin{lemma}\label{lem:nezoporneKoef}
    A Cantor real base $\BB$ satisfies Property (PF) if and only if for any string $\zz\in\F$ such that $\val(\zz)<\frac1\delta$ we have $\val(\zz)\in\Fin(\BB)$, i.e.\ $d_{\BB}(\val\zz)\in\F$.
\end{lemma}

\begin{proof}
Let us show the implication $\Leftarrow$.
Let $x,y\in{\rm Fin}(\BB)$, such that $x+y\in[0,1)$.
We need to show that $x+y\in{\rm Fin}(\BB)$.
Denote $d_{\BB}(x)=\xx=x_1x_2x_3\cdots$, $d_{\BB}(y)=\yy=y_1y_2y_3\cdots$ and $\zz=\xx\oplus\yy$. We clearly have  $\val(\zz)=x+y$.   
Since $x,y\in{\rm Fin}(\BB)$, both strings $\xx,\yy$ belong to $\F$, and thus also $\zz\in\F$. 
Suppose first that $\val(\zz)=x+y<\frac1\delta$. Then directly from the assumption we have $x+y\in\Fin(\BB)$.
Suppose now that $\frac{1}{\delta}\leq \val\zz < 1$. 
Obviously, the string $\zz' := 0^p\zz$ belongs to $\F$ and  $\val\zz' = \frac{1}{\delta}\,\val\zz < \frac{1}{\delta}$. Thus the $\BB$-expansion of $\val\zz'$ has finite support and must be of the form $0^pz_{p+1}z_{p+2}\cdots\in\F$. Using~\eqref{eq:posunrozvojealternate},
we have that $d_{\BB}(\val\zz)=z_{p+1}z_{p+2}\cdots\in\F$, i.e. $\val\zz=x+y\in{\rm Fin}(\BB)$.
    
 In order to prove $\Rightarrow$, note that Property (PF) implies that the set ${\rm Fin}(\BB)$ is closed under summation of any finite number of elements, provided the result still belongs to the interval $[0,1)$.
Let $\zz\in\F$ such that $\val(\zz)<\frac1\delta$. Since the digits of $\zz$ are non-negative, the string $\zz$ can be written as a digit-wise sum of finitely many $\BB$-admissible strings from $\F$, say 
    $$
    \zz=\zz^{(1)}\oplus \cdots \oplus\zz^{(k)}.
    $$
    As the strings $\zz^{(j)}$ are $\BB$-admissible and have finite support, we have $\val(\zz^{(j)})\in{\rm Fin}(\BB)$. By Property (PF), $\val(\zz)=\sum_{j=1}^k\val(\zz^{(j)})\in{\rm Fin}(\BB)$.
\end{proof}

In other words, Lemma~\ref{lem:nezoporneKoef} says that for each string $\zz\in\F$ with $\val(\zz)<1$, we can find a~string ${\xx}\in\F$ admissible in base $\BB$ such that $\val\xx=\val\zz$.  In the sequel, we will  repeatedly use the fact that admissible strings are lexicographically the greatest among all strings of non-negative digits representing the same value.

The following definition is inspired by~\cite{AFMP03} where similar question is studied in the simple case of $p=1$.

\begin{definition}\label{def:rewriting}
Let $\BB=({\beta_1,\beta_2,\dots,\beta_p})$ be an  alternate base and denote the corresponding $\BB^{(\ell)}$-expansions of 1 by $\t^{(\ell)}$, as in~\eqref{eq:t}. Denote the following set of strings with finite support
\begin{equation}\label{eq:S}
\S=\left\{0^{p+\ell-1}t_1^{(\ell)}t_2^{(\ell)}\cdots t_{k-1}^{(\ell)}(t_k^{(\ell)}+1)0^{\omega} \ : \  \ell\in\Z_p, k\geq 1\right\} \cup \left\{0^{p+\ell-1} \t^{(\ell)}\in \F\ : \ \ell\in\Z_p \right\}. 
\end{equation}
We say that the alternate base $\BB$ has the \textit{rewriting property}, if for any string $\aa\in\S$ there exists a~string $T(\aa)$, such that 
\begin{equation}\label{eq:transkripce}
T(\aa)\in\F, \quad \val T(\aa)=\val\aa,\quad \text{and}\quad T(\aa)\succ_{lex} \aa.  
\end{equation}
\end{definition}

The first part of the set $\S$ may contain infinitely many strings. We call them strings of Type 1. The strings in the second part of the set $\S$ are said to be of Type 2. Note that if $d_{\BB^{(i)}}\notin\F$ for every $i$, then there are no strings of Type 2. 

The following lemma shows that given an alternate base $\BB$ with the rewriting property, a non-admissible $\zz \in \F$ can always be replaced by a lexicographically greater string  in $\F$ of the same value.
%

\begin{lemma}\label{l:LEXzvetseni} 
Let $\BB$ be an alternate base with the rewriting property. Then for every ${\zz \in \F}$ non-admissible in~$\BB$ with value  $\val \zz < \frac{1}{\delta}$, there 
exist
$\xx, \yy \in \F$ and  $j \in \N$ such that $\zz = \bigl(0^{pj}\xx \bigr)\oplus \yy$ and $\xx\in\S$.
Consequently, $\bigl(0^{pj}T(\xx)\bigr)\oplus \yy \in \F$ is a $\BB$-representation of $\val\zz$ lexicographically strictly greater than $\zz$. 
\end{lemma}
\begin{proof}
Consider $\zz= z_1z_2 \cdots \in \F$ non-admissible in $\BB$. 
Since $\val \zz < \frac{1}{\delta}$, necessarily $z_1 =z_2= \dots = z_{p}=0$. 
By  Proposition~\ref{p:admisibilitaF}, there exists $i \in \N, i\geq p+1 $, such that $z_{i}z_{i+1}\cdots  \succeq_{lex} d_{\BB^{(i)}}(1)$. 

We distinguish two cases how to determine the string $\xx$:  
\begin{itemize}
\item[a)] Suppose that $z_{i}z_{i+1}\cdots = d_{\BB^{(i)}}(1)$. 

As $\zz\in\F$, we have that $d_{\BB^{(i)}}(1)\in\F$. We can take $\xx \in \S$ and $j\in\N$ such that $0^{pj}\xx =0^{i-1}d_{\BB^{(i)}}(1)$. 

\item[b)] Suppose that $z_{i}z_{i+1}\cdots  \succ_{lex} d_{\BB^{(i)}}(1)$. 

Choose minimal  $n\in \N$ such that 
  $z_{i}z_{i+1}\cdots z_{i+n} \succ_{lex} d_{\BB^{(i)}}(1)$. Then necessarily $z_{i}z_{i+1}\cdots z_{i+n-1} =  t_1^{(i)}t_2^{(i)}\cdots t_{n}^{(i)}$  and $z_{i+n} \geq   t_{n+1}^{(i)}+1$. In this case we can define $\xx \in \S$ and $j\in\N$ so that  $0^{pj}\xx = 0^{i-1} t_1^{(i)}t_2^{(i)}\cdots t_n^{(i)}\left(t_{n+1}^{(i)}+1\right)0^\omega$. 
\end{itemize}

It is obvious that the choice of $\xx$ allows to find the string $\yy$ with non-negative digits such that $(0^{jp}\xx)\oplus \yy = \zz$. 
Now replacing $0^{jp}\xx$ by $0^{jp}T(\xx)$ yields
 $\bigl(0^{pj}T(\xx)\bigr)\oplus \yy \succ_{lex} (0^{jp}\xx)\oplus\yy =\zz$.
\end{proof}

\begin{remark}\label{pozn:mene}
    From the proof of Lemma~\ref{l:LEXzvetseni}, it follows that having two strings $\aa=a_1a_2a_3\cdots$, $\aa'=a'_1a'_2a'_3\cdots\in{\S}$ such that for each $j\geq 1$ we have $a_j\leq a'_j$, it suffices to find the rewriting $T(\aa)$, the string $T(\aa')$ is then found easily.
    In practise, this substantially reduces the number of rewriting rules we need to find for the rewriting property of an alternate base.
For example if $d_{\BB^{(\ell)}}\in\F$ for every $\ell$, then one can handle only finitely many strings of Type 1. 
In particular, suppose $d_{\BB^{(\ell)}}(1) \in \F$, say  $d_{\BB^{(\ell)}}(1) = t^{(\ell)}_1t^{(\ell)}_1 \cdots t^{(\ell)}_n 0^\omega$, then for $\ell,i,k$ such that
$pk+i \geq n$ we have
$$
0^{p+\ell-1}t_1^{(\ell)}t_2^{(\ell)}\cdots t_{k-1}^{(\ell)}(t_k^{(\ell)}+1)0^{\omega} = 
0^{p+\ell-1} \t^{(\ell)} \oplus 0^{p(k+1)+\ell +i-2}10^\omega.
$$
\end{remark}

\begin{example}\label{Ex:prikladSimple}
Let us present rewriting rules 
for the base as in Example~\ref{ex:1}. 
Since 
 $d_{\BB}(1)=2010^\omega$   and  $d_{\BB^{(2)}}(1)=110^\omega$,  by Remark~\ref{pozn:mene}, it is sufficient to find the rewriting rules for a smaller number of strings, as shown in the table below. 
$$
\begin{array}{|l|l|}
\hline
\aa_1=0030^\omega & T(\aa_1)=01010^\omega\\\hline
\aa_2=00210^\omega & T(\aa_2)=01001010^\omega\\\hline
\aa_3=002010^\omega & T(\aa_3)=010^\omega\\\hline
\aa_4=00020^\omega & T(\aa_4)=00101010^\omega\\\hline
\aa_5=000110^\omega & T(\aa_5)=0010^\omega \\\hline
\end{array}
$$
Note that strings $\aa_1$, $\aa_2$, $\aa_4$ are of Type 1, whereas $\aa_3$ and $\aa_5$ are of Type 2. 
Obviously, for every $j$, the string $T(\aa_j)$ is lexicographically larger than $\aa_j$. Moreover, one can check by direct calculation that $\val{\aa_j}=\val{T(\aa_j)}$, and thus the base $\BB$ satisfies the rewriting property.

Let us  demonstrate algorithm for addition using rewriting rules. Consider for example numbers $ x, y \in [0,1)$ with  
    $d_{\BB}(x) = 000020010^\omega$  and $d_{\BB}(y) =000020^\omega$. 
    Denote $z=x+y$.  The digit-wise sum of~the~two $\BB$-expansions gives a non-admissible representation of $z$, namely $d_{\BB}(x) \oplus d_{\BB}(y) ={ 000040010^\omega}$. Using rewriting rules for strings in the set $\S$, as shown in the above table, we can write
$$
\begin{aligned}
\underbrace{ 000040010^\omega}_{non-admissible} = &\ (000010010^\omega)\oplus (\underbrace{000030^\omega}_{00\aa_1})\\
&\hspace*{3cm}\downarrow\\
&\ (000010010^\omega)\oplus (\underbrace{0001010^\omega}_{00T(\aa_1)})  =\underbrace{000111010^\omega}_{non-admissible}.
\end{aligned}
$$  
In this way we have obtained another representation of the same number $z$. Since the latter is still non-admissible, we continue by rewriting 
$$
\begin{aligned}
\underbrace{000111010^\omega}_{non-admissible} = &\ 
(000001010^\omega) \oplus (\underbrace{000110^\omega}_{\aa_5})\\
&\hspace*{3.15cm}\downarrow\\
&\ (000001010^\omega) \oplus (\underbrace{0010^\omega}_{T(\aa_5)}) 
= \underbrace{001001010^\omega}_{admissible}.
\end{aligned}
$$
The resulting digit string is admissible, and therefore it is the $\BB$-expansion of $z$, i.e. $d_{\BB}(z)=d_{\BB}(x+y) = 001001010^\omega$. 
%
\end{example}

In the above example, we have used the rewriting rules to obtain the $\BB$-expansion of $x+y$.
In general, rewriting the non-admissible string $\zz\in\F$ we need to ensure that after finitely many steps, the~procedure yields the lexicographically maximal string representing the same value as $\zz$, i.e.\ the $\BB$-expansion of $\val\zz$. This may not be always the case, as is shown already for $p=1$ in~\cite{AFMP03}.

\begin{example}
    Consider the alternate base $\BB=(\beta)$, where $\beta$ is the dominant root of the~polynomial $x^6-x^5-1$. The greedy representation of 1 is $d_{\BB}(1)=100001$. By Remark~\ref{pozn:mene}, it suffices to find the following rewriting rules.
$$
\begin{array}{|l|l|}
\hline
\aa_1=020^\omega & T(\aa_1)=1000000111110^\omega\\\hline
\aa_2=0110^\omega & T(\aa_2)=1000000011110^\omega\\\hline
\aa_3=01010^\omega & T(\aa_3)=1000000001110^\omega\\\hline
\aa_4=010010^\omega & T(\aa_4)=1000000000110^\omega\\\hline
\aa_5=0100010^\omega & T(\aa_5)=1000000000010^\omega \\\hline
\aa_6=01000010^\omega & T(\aa_6)=10^\omega \\\hline
\end{array}
$$
Again, direct calculation confirms that the base $\BB$ satisfies the rewriting property. Nevertheless, it is not sufficient to have these rewriting rules to have Property (PF).
 In fact, the base $\BB$ cannot satisfy (PF), since it does not satisfy the necessary conditions of Theorem~\ref{thm:necessary}. In particular, the number $\delta=\beta$ is neither a Pisot, nor a Salem number. For,
there exists an algebraic conjugate $\gamma$ of $\beta$, $\gamma\neq\beta$, of modulus strictly larger than~1. 
\end{example}

The following notion will be used to ensure that the procedure of rewriting a non-admissible string terminates always in finitely many steps.

\begin{definition}\label{d:weight}
We say that an alternate base $\BB$ has the \textit{weight property}, if it has the~rewriting property, and, moreover, there 
exist positive integers $w_n$, $n\geq 1$, such that the weight function $g:\F\to \N$, $g(\aa)=\sum_{n\geq 1}w_n a_n$, satisfies 
\begin{enumerate}
\item[\itshape 1.] $g(0^p\aa)=g(\aa)$ for any $\aa\in\F$;
\item[\itshape 2.] $g(\aa)\geq g(T(\aa))$ for any $\aa\in\S$.
\end{enumerate}
\end{definition}

The weight property is sufficient to guarantee the positive finiteness property. We will use the obvious fact that the weight function satisfies $g(\xx\oplus\yy)=g(\xx)+g(\yy)$ for any two digit strings $\xx,\yy\in\F$. 

\begin{theorem}\label{thm:sufficient}
 Let $\BB=({\beta_1, \beta_2,\dots,\beta_p})$ be an alternate base satisfying the weight property. Then $\BB$ has Property (PF).
 \end{theorem}

For the proof of Theorem~\ref{thm:sufficient}, we need an auxiliary statement.

\begin{lemma}\label{l:topo} 
Let $g: \F \mapsto \N$ be a weight function and let $\bigl(\aa^{(k)}\bigr)_{k \geq 1}$ be a sequence of infinite strings satisfying  for every $k \geq 1$
\begin{itemize}
    \item $\aa^{(k)} \in \F$;
    \item $\aa^{(k)} \prec_{lex} \aa^{(k+1)}$. 
\end{itemize} 
Then  the integer sequence  $\bigl(g(\aa^{(k)})\bigr)_{k \geq 1}$ is not bounded. 
\end{lemma}
\begin{proof}  Let $w_1, w_2, \ldots$ be positive integers and  $g(\aa) = \sum_{n\geq 1 }w_na_n$ for every sequence $\aa= (a_n)_{n \geq1} \in \F$. 

We proceed by contradiction. Assume that there exists
$H \in \mathbb{N}$ such that $g\big(\aa^{(k)}\big) \leq H$ for every $k \geq 1$.
Since all coefficients $w_n$ are positive integers, we derive
\begin{equation}\label{omezeny}
a_n^{(k)}\in \{0,1,\ldots, H\}\qquad \text{and}\qquad \#\supp \aa^{(k)} \leq  H \qquad \text{  for all  } n, k \geq 1.  
\end{equation}
 The set $\{0,1,\ldots, H\}^{\N}$ equipped with the product topology is a compact space and thus the increasing sequence  $\big(\aa^{(k)}\big)_{k \geq 1}$ has a limit, say  $\bb = \lim_{k\to +\infty}\aa^{(k)}$. 
 Obviously, $\aa^{(k)} \preceq_{lex} \bb$ for each $k \geq 1$. 

Let us at first show that
$\supp \bb $ is infinite. Suppose the contrary, i.e.\ that $\bb = b_1b_2\cdots b_N 0^\omega$ for some $N \geq 1 $. Since $\bb = \lim_{k\to +\infty}\aa^{(k)}$, there exists $k_0 \geq 1$ such that  $\aa^{(k)}$ has a prefix $b_1b_2\cdots b_N$ for each $k>k_0$.  
The inequality $\aa^{(k)} \preceq_{lex} \bb  =  b_1b_2\cdots b_N 0^\omega$ implies $\aa^{(k)} = \bb$ for all $k > k_0$, and that is a~contradiction with the fact that $\big(\aa^{(k)}\big)_{k\geq1}$ is strictly increasing. 

Since $\supp \bb $ is infinite, we can choose $M\geq 1$ such that the set
$$
S_M := \{n \geq 1: b_n \neq 0 \text{\ and \ }  n \leq  M\} \subset \supp \bb
$$  
has cardinality  $\#S_M > H $.
Since for all sufficiently large $k$ the string $\aa^{(k)}$ has prefix $b_1b_2 \cdots b_M$, it has to be $S_M \subset \supp \aa^{(k)}$, and that is a contradiction with \eqref{omezeny}.
\end{proof}

\begin{proof}[Proof of Theorem \ref{thm:sufficient}]   
By Lemma \ref{lem:nezoporneKoef}, it is sufficient to show that if $\zz  \in \F$  and $z =\val \zz <\frac1\delta$, then  the~$\BB$-expansion of~$z=\val \zz $  belongs to $\F$ as well. 
If $\zz$ is  admissible in base $\BB$, there is nothing left to discuss. On the other hand, if $\zz$ is not  admissible in base $\BB$, then Lemma \ref{l:LEXzvetseni} allows us to find a lexicographically greater $\BB$-representation of $z$, which also has a finite support. 
Let us show that after finitely many applications of Lemma \ref{l:LEXzvetseni}  we get the~$\BB$-expansion of $z$. 

Let us proceed by contradiction. Assume that applying Lemma \ref{l:LEXzvetseni} repeatedly yields an infinite sequence of strings
$\zz=\zz^{(0)}, \zz^{(1)}, \zz^{(2)}, \ldots $ such that for every $n \in \N$
$$
 \zz^{(n)} \ \text{ is a $\BB$-representation of $z$}, \quad
\zz^{(n)} \ \text{  is not admissible in $\BB$,} \quad \text{and}\quad 
\zz^{(n)} \prec_{lex} \zz^{(n+1)}. 
$$

Let $g$ be a weight function.
To obtain $ \zz^{(n+1)}$ we use Lemma~\ref{l:LEXzvetseni}, i.e., if  $\zz^{(n)} =  \bigl(0^{pj}\xx \bigr)\oplus \yy$, then  $ \zz^{(n+1)} =   \bigl(0^{pj}T(\xx) \bigr)\oplus \yy$. 

Properties of $g$ guarantee that 
$g\big(\zz^{(n+1)}\big)\leq g\big(\zz^{(n)}\big)$ for all $n \in \N$. 
In particular, the  values ${g\big(\zz^{(n)}\big) \in \N}$ are bounded by $g\big(\zz^{(0)}\big) = g(\zz)$, and that is a contradiction with Lemma~\ref{l:topo}. 
\end{proof}


The above theorem gives us a sufficient condition for Property (PF). When searching for bases with Property (F), by  Proposition~\ref{l_nec_F_1}, 
it suffices to limit our considerations to simple Parry alternate bases.
The following proposition shows that for such bases Properties (F) and (PF) are equivalent.


\begin{proposition}\label{p:PF=F}
Let $\BB=({\beta_1, \beta_2,\dots,\beta_p})$ be a simple Parry alternate base. Then $\BB$ has (F) if and only if $\BB$ has (PF). 
\end{proposition}

We first prove an auxiliary statement.

\begin{lemma}\label{l:claim}
Let $\BB$ be a simple Parry alternate base.
For all $j,k\in\N$, $j\geq 1$, there exists a $\BB$-representation of the number $\val (0^{j-1}10^\omega)$ of the form  $0^{j-1}u_j u_{j+1}u_{j+2} \cdots \in \F$ such that $u_{j+k}\geq 1$.
\end{lemma}

\begin{proof}
We proceed by induction on $k$.
For $k=0$, the statement is trivial.
Let $k\geq1$. By induction hypothesis, there exists a string $\uu=0^{j-1}u_ju_{j+1}u_{j+2}\cdots\in\F$ 
such that 
$\val (0^{j-1}10^\omega) = \val\uu$ and ${u_{j+k}\geq 1}$. 
From the definition of $\t^{(j+k+1)}$, we derive that the~string
${\mathbf v}=0^{j+k-1}(-1)\t^{(j+k+1)}$ of integer digits has value~0, i.e.\ $\val {\mathbf v} = 0$.
The string $\uu\oplus {\mathbf v}$ has all digits non-negative and has finite support. Its digit at position $j+k+1$ is equal to $u_{j+k+1}+t_1^{(j+k+1)}\geq 1$. 
\end{proof}

\begin{proof}[Proof of Proposition~\ref{p:PF=F}]
Clearly, if $\BB$ satisfies (F), then it satisfies (PF). For the opposite implication, 
assume that a simple Parry alternate base $\BB$ has Property (PF). 
Therefore, according to Lemma~\ref{lem:nezoporneKoef}, if a~number  $z \in [0,1)$ has a $\BB$-representation  in $\F$, then the~$\BB$-expansion of $z$ belongs to $\F$ as well. 

In order to show Property (F), it is thus sufficient to verify that for any $x,y\in\Fin(\BB)\cap[0,1)$ such that $z:=x-y\in(0,1)$ we can find a  $\BB$-representation of $z$ with finite support.
Let $\xx,\yy\in\F$ be the~$\BB$-expansions of numbers $x,y\in[0,1)$ such that $x=\val\xx >  y=\val\yy >0$. 

We proceed by induction on the sum of digits in $\yy$. If the sum is 0, then $y=0$ and the statement is trivial. Suppose the sum of digits in $\yy$ is positive.
Since $\val\xx > \val\yy$, there exist  $j, k \in \N$, $j\geq 1$, such that $\xx = \xx' \oplus \xx''$  and   $\yy = \yy' \oplus \yy'' $,  where $\xx'' =  0^{j-1}10^\omega$ and $\yy'' = 0^{j+k-1}10^\omega $.  
It follows directly from Lemma~\ref{l:claim} that $\val{\xx}'' - \val{\yy}''$ has a finite $\BB$-representation, say $\zz'' \in \F$.
Hence 
$$
z=\val\xx -  \val\yy = \val\xx' +  \val \zz'' -  \val\yy'.
$$
Property (PF) guarantees 
that  $\val \xx' +  \val \zz''$ has a finite $\BB$-expansion, say  $\xx_{new}$.  Thus
$$
z=\val\xx -  \val\yy = \val\xx_{new} -  \val\yy'.
$$
The sum of digits in $\yy'$ is smaller by 1 than the sum of digits in $\yy$. Induction hypothesis implies that $z$ has a finite $\BB$-representation, and, by Property (PF), also a finite $\BB$-expansion. 
\end{proof}

As a consequence of Proposition~\ref{p:PF=F}, we formulate a sufficient condition for Property~(F).

\begin{theorem}\label{t:sufficientF}
      Let $\BB=({\beta_1, \beta_2,\dots,\beta_p})$ be a simple Parry alternate base satisfying the~weight property. Then $\BB$ has Property (F).
\end{theorem}

\section{A class of bases with finiteness property}\label{sec:nerovnosti}

In this section, we provide the proof of Theorem~\ref{thm:nerovnosti}. In particular, we show that an alternate base $\BB$ for~which the sequences $\t^{(\ell)}$ satisfy the set of inequalities~\eqref{eq:GFS}, has Property (PF).  
This will be shown by verifying that a base $\BB$ satisfying~\eqref{eq:GFS} has the weight property, and thus satisfies the assumptions of~Theorem~\ref{thm:sufficient}.
Firstly, we show the rewritings of the strings in the set $\S$ required by Definition \ref{def:rewriting}.



\begin{proposition}\label{typ1}  Let $\BB$ be an alternate base and 
 $d_{\BB^{(\ell)}} = \t^{(\ell)}$ satisfy \eqref{eq:GFS}.  Then ${\BB}$ is a~Parry alternate base and  has the rewriting property. 
\end{proposition}

\begin{proof} 
Inequalities~\eqref{eq:GFS} imply that for every  $\ell \in \Z_p$, the sequence $\Big(t^{(\ell -n)}_{n+1}\Big)_{n\in \N}$ is non-increasing, and thus eventually constant. 
Let us denote the constant $s_\ell$, i.e.\ we have $t^{(\ell-n)}_{n+1}=s_{\ell}$ for all sufficiently large $n$. Therefore, we can write for some  $m \in \N$ that
\begin{equation}
\label{eq:opakujise}\t^{(\ell)} = t^{(\ell)}_1 \cdots  t^{(\ell)}_{pm} (s_{\ell} s_{\ell+1}\cdots s_{\ell+p-1})^\omega \quad \text{for every }  \ell\in  \Z_p,
\end{equation}
where the index of $s$ is considered $\!\!\!\mod p$.
In other words, inequalities~\eqref{eq:GFS} imply, that $\BB$ is a Parry alternate base. Moreover, either $\BB$ is a simple Parry base (if $s_{\ell}=0$ for every $\ell\in\Z_p$), or no $\t^{(\ell)}$ has finite support. 

In order to show the rewriting property, we have to find, for each string $\aa\in\S$, a string $T(\aa)$ satisfying the requirements~\eqref{eq:transkripce}, i.e.
$$
a) \ \ T(\aa) \in \F,\qquad b) \ \ \val\aa = \val  T(\aa) \text{  \ \ \ and \ \ \ } c) \ \ \aa \prec_{lex} T(\aa) . 
$$
The proof will be divided according to the form of the string $\aa\in\S$. The strings are of two types: 

\noindent
\textbf{Type 1:} For $\ell,i \in \Z_p$ and $k \in \N$
\begin{equation*}
\begin{aligned}
    \xx_{\ell, i, k} &= 0^{p+\ell-1}  t_1^{(\ell)}t_2^{(\ell)} \cdots  t_{pk+i-1}^{(\ell)} \left(t_{pk+i}^{(\ell)}+1\right)0^\omega,\\
    T(\xx_{\ell, i, k}) &= 0^{p+\ell-2}1 0^{pk+i} \left(t_1^{(\ell+i)}-t_{pk+i+1}^{(\ell)}\right)\left(t_2^{(\ell+i)}-t_{pk+i+2}^{(\ell)}\right) \left(t_3^{(\ell+i)}-t_{pk+i+3}^{(\ell)}\right)\cdots. 
\end{aligned}
\end{equation*}
\textbf{Type 2:} For $\ell\in\Z_p$
\begin{equation*}
    \xx_\ell = 0^{p+\ell -1}\t^{(\ell)}, \qquad
    T(\xx_\ell) = 0^{p+\ell-2}1 0^\omega. 
\end{equation*}

Realize that if the periodic part in~\eqref{eq:opakujise} is non-vanishing $(s_1s_2 \cdots s_p)^\omega \neq 0^\omega$,  i.e.\,$\BB$ is not a simple Parry base, then $\S$ contains only strings of Type 1.

\medskip
\noindent\textbf{Type 1:}
For $\xx_{\ell, i, k}$, requirement c) is trivially satisfied. In order to show a), we use \eqref{eq:opakujise}.  As $t^{(\ell) }_{n} = s_{\ell+n-1}$ for every  $n \geq pm+1$, we have 
$$
t^{(\ell +i)}_{n} -  t^{(\ell)}_{pk+i+n} =
s_{\ell+i+n-1} - s_{\ell+pk +i+n-1} = 0 \quad \text{for every $n \geq pm+1$}.
$$
Consequently, 
\begin{equation}\label{eq:konecne} T(\xx_{\ell, i, k}) =0^{p+\ell-2}1 0^{pk+i} \left(t_1^{(\ell+i)}-t_{pk+i+1}^{(\ell)}\right)\left(t_2^{(\ell+i)}-t_{pk+i+2}^{(\ell)}\right) \cdots \left(t_{pm}^{(\ell+i)}-t_{pm+pk+i}^{(\ell)}\right)0^\omega \in \F,
\end{equation} 
which proves a).

In order to show b) we need to verify \begin{equation}\label{eq:konecne2}
  W_{p(k+1)+i +\ell -1} + \sum_{j=1}^{pk+i}t_{j}^{(\ell)}W_{j+p+\ell-1} =  W_{p+\ell-1}  +\sum_{j\geq 1} \bigl( t^{(\ell+i)}_j - t^{(\ell)}_{pk+i+j}\bigr)W_{j+p(k+1) +\ell +i-1}\,,
\end{equation}
where $W_n = \left(\prod_{j=1}^n{\beta_j}\right)^{-1}$ for all $n\in \N$, $n\geq 1$. 
Equivalently, 
\begin{equation}\label{w}
     \sum_{n\geq 1}W_{n+p+\ell-1}\,t^{(\ell)}_{n} - W_{p+\ell-1}
     = \sum_{n\geq 1} W_{n+p(k+1) +\ell +i-1}\,t^{(\ell + i)}_{n}
     - W_{p(k+1)+i +\ell -1} \,.
\end{equation} 
The definition of $\Bigl( t_n^{(r)}\Bigr)_{n\geq1}$ implies  $1=\sum_{n\geq 1}W'_{n}t^{(r)}_{n}$,   where  $W'_n = \Bigl(\prod_{j=r}^{n+r-1}{\beta_j}\Bigr)^{-1}\!\!= \frac{W_{n+r-1}}{W_{r-1}}$.
Hence 
\begin{equation}\label{eq:wPomocne}
W_{r-1} = \sum_{n\geq1}W_{n+r-1}t^{(r)}_{n}. 
\end{equation}
Using this fact  with $r=p+\ell$  on the left-hand side of~\eqref{w} yields 0. Similarly, using~\eqref{eq:wPomocne} with  $r =p(k+1)+ \ell+i$ gives 0 on the right-hand side of~\eqref{w}.  

\medskip
\noindent\textbf{Type 2:}
If $\BB$ is a simple Parry base, all three requirements a), b), c) for strings of the form $\xx_\ell$ are satisfied trivially. 
\end{proof}

\begin{example}
In the notation of the proof of Proposition~\ref{typ1}, the strings $\aa_1,\dots,\aa_5$ from Example~\ref{Ex:prikladSimple} are
$$
\aa_1=\xx_{1,1,0},\quad 
\aa_2=\xx_{1,2,0},\quad 
\aa_3=\xx_1,\quad
\aa_4=\xx_{2,1,0},\quad
\aa_5=\xx_2,
$$
and the rewritings $T(\aa_j)$ given in
Example~\ref{Ex:prikladSimple} are obtained as described in the proof. 
\end{example}

In order to prove Theorem~\ref{thm:nerovnosti}, it remains to show that~\eqref{eq:GFS} implies the existence of a~weight function for $\BB$.  The weight function must satisfy certain inequalities for every string in $\S$.

 Recall that \eqref{eq:GFS} implies that $\BB$ is a Parry base. In particular, $\t^{(\ell)} = t^{(\ell)}_1 \cdots  t^{(\ell)}_{mp} (s_{\ell} s_{\ell+1}\cdots s_{\ell+p-1})^\omega$, see  \eqref{eq:opakujise}. For $\ell \in \mathbb{Z}_p$,  denote 
$$
T^{(\ell)}_j=\sum_{{\substack{1\leq n \leq mp \\ n=j \!\!\!\!\mod p}}}\!\!\!\!\!\!t_n^{(\ell)} \qquad \text{for all } j \in \Z_p=\{1,2,\dots,p\}.
$$

Note that  $T_{j}^{(\ell)} \in \N$   and,  moreover,  $T_{1}^{(\ell)} \geq t_1^{(\ell)}\geq 1$.  Inequalities \eqref{eq:GFS} imply that for all $\ell \in \mathbb{Z}_p$ 
\begin{equation}\label{Tecka}
    T^{(\ell )}_1 \geq      T^{(\ell -1 )}_2 \geq  T^{(\ell -2 )}_3 \geq  \cdots \geq    T^{(\ell -p+2)}_{p-1}  \geq    T^{(\ell -p+1)}_p  . 
\end{equation}

\noindent Now consider matrices  $I,  P, K \in \N^{p\times p}$, where  $I$ is the identity matrix, $P$ is the permutation matrix
\begin{equation}\label{eq:maticeK}
P=
\begin{pmatrix} 
    {\Theta} & I_1 \\
    I_{p-1} & \Theta
    \end{pmatrix},
\qquad  \text{and} \qquad 
K_{\ell , j} = T^{(\ell+1)}_{j-\ell}\,,\quad \text{for all \ \ } \ell, j \in \{1,2,\ldots, p\}.
\end{equation}
The entries of the permutation matrix $P$ can also by written as $P_{\ell,j}=\delta_{\ell,j+1}$ where $\delta_{i,j}$ is the Kronecker symbol whose indices are considered $\!\!\!\mod p$, as well as both indices of $T^{(\ell+1)}_{j-\ell}$ in the above relation.

In the sequel, we adopt the convention that the $j$-th component of a vector $\vec{u}$ is denoted by $(\vec{u})_j$ and that for vectors $\vec{u},\vec{v}$ we write $\vec{u}\leq \vec{v}$ if $(\vec{u})_j\leq (\vec{v})_j$ for every $j$.

\begin{lemma}\label{existujeU} 
There exist positive integers $u_1,u_2,\ldots, u_p$ such that for the vector  $\vec{u} = (u_1,u_2,\ldots, u_p)^T$ it holds that
$$ 
(I-P)(K-I)\vec{u} = \vec{0}\quad \text{and} \quad K\vec{u}\geq \vec{u} ,
$$
where by the vector inequality we mean inequality in each component.
\end{lemma}

\begin{proof} 
First, we show that all non-diagonal elements of the matrix 
$M=(I-P)(K-I)$ 
are non-negative. Indeed, if $\ell\neq j$, then, considering the indices 
$\!\!\!\mod p$, we can write 
$$
M_{\ell, j} = K_{\ell, j} - K_{\ell -1, j} +\delta_{\ell, j+1} = \underbrace{ T_{j-\ell}^{(\ell+1)} - T_{j-\ell+1}^{(\ell)}}_{\geq 0} +\underbrace{\delta_{\ell, j+1}}_{\geq 0}\geq 0,
$$
where we have used inequalities \eqref{Tecka}.
Moreover, $M_{\ell+1, \ell}\geq \delta_{\ell+1, \ell +1 }= 1$ for all $\ell \in \Z_p=\{1,2,\ldots,p\}$.
Therefore, there exists a positive integer $c$ such that the matrix $M+cI$ is non-negative and irreducible. 
The row vector $\vec{1} = (1,1,\ldots, 1)$ 
has property $\vec{1}(I-P) = \vec{0}$, and thus $\vec{{1}}M = \vec{0}$. Therefore, $\vec{{1}}$ is a positive left eigenvector of $M+cI$ corresponding to the eigenvalue $c$. 
According to Perron-Frobenius theorem, $c\in \N$ is equal to the spectral radius of $M+cI$. Therefore, there exists a positive column vector $\vec{u}= (u_1,u_2, \ldots, u_p)^T$ such that $(M+cI) \vec{u} = c\vec{u}$. 
Since $M+cI$  has integer components, the eigenvector $\vec{u}$ corresponding to $c$ can be chosen to have integer components. 
Moreover, $M\vec{u}= \vec{0}$. 

To prove the second part of the lemma, i.e.\ $K\vec{u}\geq \vec{u}$, we denote $\vec{\varepsilon} = (K-I)\vec{u}=(\varepsilon_1,\dots,\varepsilon_p)^T$.  Then 
$$
 \vec{0} = M\vec{u} = (I-P) (K-I)\vec{u}  =  (I-P) \vec{\varepsilon} = (\varepsilon_1-\varepsilon_p, \varepsilon_2 - \varepsilon_1,  \varepsilon_3 - \varepsilon_2,\ldots, \varepsilon_p -\varepsilon_{p-1})^T\,.
$$
 Hence all the components of the vector  $\vec{\varepsilon} = (K-I)\vec{u}$ have the same value, or, equivalently,  
 \begin{equation}\label{eq:epsilon}
 \vec{\varepsilon} = \kappa (1,1, \ldots , 1)^T\quad\text{ for some }\kappa \in \mathbb{R},
 \end{equation}
 and thus
\begin{equation}
    \label{kapa}
    K \vec{u} = \vec{u} + \kappa  (1,1, \ldots , 1)^T.
\end{equation}
In order to prove the inequality $K\vec{u}\geq  \vec{u}$ it suffices to show that $\kappa \geq 0$.
Recall that all elements of $K$ are non-negative
 and $K_{\ell, \ell+1}= T_{1}^{(\ell+1)}\geq  t_{1}^{(\ell+1)}\geq 1$ for all $\ell\in \mathbb{Z}_p$.  
 In~other words,  each column of $K$ contains at least one element $\geq 1$.  
 Consequently, $\vec{{1}} \leq \vec{{1}}\,K$, and considering \eqref{kapa} we obtain
$$ 
\vec{{1}}\, \vec{u}\leq \vec{{1}}\,K \,\vec{u} = \vec{{1}}\, \bigl(\vec{u} + \kappa (1,1,  \ldots, 1)^T\bigr) =  \vec{{1}}\, \vec{u} + p \kappa. 
$$ 
Necessarily $\kappa \geq 0$.
\end{proof}

\begin{proposition}\label{c:weight}
 Let $\BB$ be an alternate base satisfying~\eqref{eq:GFS}. Then $\BB$ satisfies the~weight property.   
\end{proposition}

\begin{proof}
We now find a weight function $g$ satisfying requirements of Definition~\ref{d:weight}. 
In order to define coefficients $w_1,w_2,  w_3, \ldots$ for the weight function $g$, we use positive integers $u_1, u_2, \ldots, u_p$ found in~Lemma~\ref{existujeU} as follows. We set
$$
w_n:=u_i \quad\text{if} \quad i = n\!\! \mod p. 
$$

Let us verify requirements of Definition~\ref{d:weight}.
The fact that $g(0^p\xx) = g(\xx)$ follows directly from the~definition of~$(w_n)_{n\geq 1}$ as a sequence with period $p$.

The second requirement, namely that $g(\aa)\geq g(T(\aa))$ for each string $\aa\in\S$, will be proven separately for the two types of strings as distinguished in the proof of Proposition~\ref{typ1}.
 Let us first focus on  strings of Type 1. Inequality  $g\bigl(T(\xx_{\ell, i, k})\bigr)  \geq g(\xx_{\ell, i, k}) $  has the~form \begin{equation}\label{eq:coChceme}
  w_{p(k+1)+i +\ell -1} + \sum_{j=1}^{pk+i}t_{j}^{(\ell)}w_{j+p+\ell-1} \geq   w_{p+\ell-1}  +\sum_{j=1}^{mp} \bigl( t^{(\ell+i)}_j - t^{(\ell)}_{pk+i+j}\bigr)w_{j+p(k+1) +\ell +i-1}\,,
\end{equation}
or equivalently (using periodicity of $w_n$)
 \begin{equation}\label{eq:coChceme2}
  w_{i +\ell -1} + \sum_{j=1}^{(m+k)p+i}t_{j}^{(\ell)}w_{j+\ell-1} \geq   w_{p+\ell-1}  +\sum_{j= 1}^{mp}  t^{(\ell+i)}_j w_{j+\ell +i-1}\,,
\end{equation}
The right-hand side of the inequality~\eqref{eq:coChceme2} does not depend on $k$, whereas the left-hand side is increasing with increasing $k$. 
Since 
$$
\sum_{j=1}^{(m+k)p+i}t_{j}^{(\ell)}w_{j+\ell-1} \geq \sum_{j=1}^{mp}t_{j}^{(\ell)}w_{j+\ell-1}, 
$$
for validity of~\eqref{eq:coChceme2}, it suffices to demonstrate that for every $i, \ell\in \Z_p$ the following holds
 \begin{equation}\label{eq:coChceme3}
  w_{i +\ell -1} + \sum_{j=1}^{mp}t_{j}^{(\ell)}w_{j+\ell-1} =   w_{\ell-1}  +\sum_{j= 1}^{mp}  t^{(\ell+i)}_j w_{j+\ell +i-1}\,.
\end{equation}

In order to work with the latter equality, we will denote the $j$-th component of some vector $\vec{z}$ as $(\vec{z})_j$ and use the definition of the matrix $K$,  see~\eqref{eq:maticeK}. In this notation, we have
$$ 
\begin{aligned}
\sum_{j=1}^{mp}t_{j}^{(\ell)}w_{j+\ell-1} &= 
\sum_{r \in \Z_p}\Bigl( w_{r+\ell-1}  \!\!\!\!\!\!\sum_{{\substack{1\leq j \leq mp \\ j=r \!\!\!\!\mod p}}}t_{j}^{(\ell)}\Bigr) = \sum_{r \in \Z_p} w_{r+\ell-1} \ T^{(\ell)}_r=\\ 
&= \sum_{r \in \Z_p} w_r \ T^{(\ell)}_{r-\ell+1}=  \sum_{r \in \Z_p} w_r K_{\ell-1,r} =  (K \vec{u})_{\ell-1},   
\end{aligned}
$$
where $\vec{u}=(w_1,w_2,\dots,w_p)^T$.
By analogous considerations, we derive that $$
\sum\limits_{j= 1}^{mp}  t^{(\ell+i)}_j w_{j+\ell +i-1} = (K \vec{u})_{\ell+i-1}.
$$ 
Recall the permutation matrix $P$ from~\eqref{eq:maticeK}.
Since for each coordinate index $j \in \Z_p$  and arbitrary vector $\vec{z}$ we have $(\vec{z})_{j+1} = (P\vec{z})_j$, the equality~\eqref{eq:coChceme3} can be rewritten
$$
0=-(P^i\vec{u})_{\ell-1} - (K \vec{u})_{\ell-1} +  \vec{u}_{\ell-1} + (P^iK \vec{u})_{\ell-1}
=
\bigl((I-P^i)(I-K)\vec{u}\bigr)_{\ell-1}  
$$
As $(I-P^i) = (I+P+P^2+\cdots + P^{i-1})(I-P)$, the last equality, and hence also inequality~\eqref{eq:coChceme}, is a consequence of Lemma~\ref{existujeU}.  

If the base $\BB$ is non-simple Parry, i.e.\ sequences $\t^{(j)}$, $j\in\Z_p$, have infinitely many non-zero digits, then the proof is finished, since no string of the set $\S$ is of Type 2. We continue the discussion for the case of  a simple Parry base $\BB$, where $\t^{(\ell)} = t^{(\ell)}_1t^{(\ell)}_2\cdots t^{(\ell)}_{pm} 0^\omega$ for all $\ell\in \mathbb{Z}_p$.  The inequality  
$g(\xx_{\ell})   \geq   g\bigl(T(\xx_{\ell}) \bigr) $  reads
$$
\sum_{j=1}^{mp}t^{(\ell)}_jw_{p+\ell +j-1} \geq  w_{p+\ell -1}. 
$$
As we have shown above, this is equivalent to 
$ (K \vec{u})_{\ell-1} \geq (\vec{u})_{\ell-1}$. This follows from $K\vec{u}\geq \vec{u}$, as stated in Lemma~\ref{existujeU}. 
\end{proof}

\begin{proof}[Proof of Theorem~\ref{thm:nerovnosti}]
Proposition~\ref{c:weight} combined with Theorem~\ref{thm:sufficient} yields the first part of Theorem~\ref{thm:nerovnosti}, i.e.\ that bases $\BB$ satisfying inequalities~\eqref{eq:GFS} have (PF). 
As for the second claim, we use Proposition~\ref{p:PF=F} to derive that simple Parry bases $\BB$ satisfying~\eqref{eq:GFS} have property (F). 
\end{proof}

It follows from the proof of Proposition~\ref{c:weight} that for $p=1$, i.e.\ in the case of Rényi numeration systems, the weight function can always be taken constant, say $w_n=1$ for $n \in \N$. This is no longer the case for $p>1$, as we illustrate on the base from Example~\ref{Ex:prikladSimple}.  

\begin{example}
Consider the base  $\BB=(\frac{1+\sqrt{13}}{2},\frac{5+\sqrt{13}}{6})$ from Examples~\ref{ex:1} and~\ref{Ex:prikladSimple}.
The~expansions of~1, namely  
 $d_{\BB}(1)=2010^\omega$   and  $d_{\BB^{(2)}}(1)=110^\omega$, satisfy inequalities~\eqref{eq:GFS}, and so by Proposition~\ref{c:weight}, the base $\BB$ has the weight property. 
 The weight function in this case cannot be chosen to have constant weights $w_j=w$ for all $j$. 
 In particular, we can derive that the matrix $M=(I-P)(K-I)$ from Lemma~\ref{existujeU} is of the form $M=\left(\begin{smallmatrix}
 -3 & 2\\
 3 & -2
 \end{smallmatrix}\right)$, and it has the positive eigenvector $\vec{u}=(2,3)^T$. Thus we choose $w_1=2$, $w_2=3$. 
The values of the~weight function $g$ evaluated at the strings of the set $\S$ and their rewritten forms are shown in the~table below.
$$
\begin{array}{|rcl|c|l|c|}
\hline
 &\aa& & g(\aa) & \quad T(\aa)& g(T(\aa))\\\hline
\xx_{1,1,0} &=&0030^\omega & 6  & \ 01010^\omega & 6\\\hline
\xx_{1,2,0} &=&00210^\omega & 7 & \ 01001010^\omega & 7\\\hline
\xx_1 &=&002010^\omega & 6 &\ 010^\omega & 3\\\hline
\xx_{2,1,0} &=&00020^\omega & 6 &  \ 00101010^\omega & 6\\\hline
\xx_2 &=&000110^\omega & 5 & \ 0010^\omega & 2 \\\hline
\end{array}
$$
Note that a constant weight function is not suitable, since the sum of digits in the string $\xx_{2,1,0}=00020^\omega$ is strictly smaller than the sum of digits in its rewritten form $T(\xx_{2,1,0})=00101010^\omega$.
\end{example}

\begin{remark}
    For $p=1$ it was shown in~\cite{Akiyama06} that the only bases $\beta$ satisfying Property (PF) without having (F) are the ones for which $d_{\BB}(1)=t_1t_2\cdots t_{m-1}t_m^\omega$ with  $t_1\geq t_2\geq\cdots\geq t_{m-1}\geq t_m\geq 1$. 
    In other words, the bases that satisfy inequalities~\eqref{eq:GFS} but are not simple Parry. 
    It~is an open question whether similar statement can be expected also for $p>1$.
\end{remark}

\section{Comments}

Some sources focussing on the finiteness property in the Rényi numeration system extend the notion of a $\beta$-expansion for numbers outside of the interval $[0,1)$. One needs to simply realize that multiplying a number by the base corresponds to shifting of the radix point. 
Similar considerations are valid also for alternate bases $\BB=(\beta_1,\dots,\beta_p)$ with the limitation that only shift by a multiple of the period $p$ is allowed, which correspond to multiplication  of the number by a power of $\delta=\prod_{i=1}^p\beta_i$.
Then we can define
$$
{\rm fin}(\BB)=\pm\bigcup_{k\in\N} \delta^k {\rm Fin}(\BB).
$$

When $p=1$, such a defition allowed to formulate the finiteness property as requiring that 
${\rm fin}(\beta)=\Z[\beta,\beta^{-1}]$, which is equivalent to saying that ${\rm fin}(\beta)$ has the algebraic structure of a ring, i.e.\ it is closed under addition, subtraction and multiplication.
For $p=1$, closedness under addition and subtraction implies closedness under multiplication.
The same implication does not hold for alternate bases with $p>1$.
Our definition of (F) implies closedness of ${\rm fin}(\BB)$ under addition and subtraction. However, this does not imply that  ${\rm fin}(\BB)$  is a~ring, as it is shown on the following example.

\begin{example}
Consider $\BB=({\beta_1,\beta_2})$, where $\beta_1$ is the positive root of the polynomial $2x^2-7x-3$, and $\beta_2$ is the positive root of the polynomial $3x^2-5x-4$. Then $d_{\BB}(1)=32$, $d_{\BB^{(2)}}(1)=21$, and, by Theorem~\ref{thm:nerovnosti}, the system has Property (F). Now consider the number $x=\frac{1}{\beta_1}$. Surely, $x \in {\rm fin}(\BB)$, but
\begin{equation*}
    d_{\BB}(x^2)=0(02)^{\omega},
\end{equation*}
thus ${\rm fin}(\BB)$ is not closed under multiplication of its elements.
\end{example}

Additional requirement on the alternate base is needed to ensure that ${\rm fin}(\BB)$ is a ring. 

\medskip
Frougny and Solomyak~\cite{FS92} study another arithmetical characteristic of the Rényi numeration system with base $\beta$, namely the quantities $L_{\oplus}$ which denote the bound on the increase of the length of the fractional part arising when adding numbers with finite expansion. They have shown this quantity is finite for a Pisot base $\beta$.
Later Bernat~\cite{Bernat07} has shown this is true even for Perron numbers.
Several papers~\cite{AFMP03,GaGa04,GMP04} are devoted to methods of computation 
of $L_\oplus$ and analogous value $L_\otimes$ for multiplication.
For alternate bases, the problem remains to be investigated.






\section*{Acknowledgements}

The research was supported by the European Union within the Project CZ.02.1.01/0.0/0.0/16\_019/0000778 and by the Czech Technical University grant SGS23/187/OHK4/3T/14.

\bibliographystyle{siam}
\bibliography{references}

\end{document}